\documentclass[12pt]{article}
\usepackage{amsmath,amsfonts,amssymb,amsthm,color}

\usepackage{amsmath, amsfonts, amssymb, amsthm, graphicx, color,caption}
\usepackage{epsfig}
\usepackage{fancyhdr}
\usepackage{framed}
\usepackage{floatrow}
\usepackage{mdframed}
\usepackage{url}
\usepackage{cfr-lm}
\usepackage{color}

\usepackage{imakeidx}
\makeindex

\newcounter{abci}\renewcommand{\theabci}{\alph{abci}}

\theoremstyle{plain}
\newtheorem{lem}{Lemma}
\newtheorem*{lem*}{Lemma}
\newtheorem{thm}{Theorem}
\newtheorem{prop}{Proposition}
\newtheorem*{prop*}{Proposition}
\newtheorem{assum}{Assumption}
\newtheorem{Def}{Definition}
\newtheorem{cor}{Corollary}
\newtheorem*{cor*}{Corollary}

\theoremstyle{definition}
\newtheorem{ex}{Example}

\def \cF {\mathcal{F}}
\def \cI {\mathcal{I}}
\def \cS {\mathcal{S}}

\def \cL {\mathcal{L}}
\def \cT {\mathcal{T}}
\def \cR {\mathcal{R}}

\def\cE{\mathcal{E}}

\def\S{{\mathcal{Z}}}

\def \S {\mathcal{S}_t}

\def\be{\begin{equation}}
\def\ee{\end{equation}}

\title{Invariant Galton-Watson trees: metric properties and attraction with respect to generalized dynamical pruning}

\author{Yevgeniy Kovchegov\footnotemark[1], 
Guochen Xu\footnotemark[1], 
and 
Ilya Zaliapin\footnotemark[2]}
\date{}

\begin{document}

\maketitle

\begin{abstract}
{\it Invariant Galton-Watson (IGW)} tree measures is a one-parameter family of 
critical Galton-Watson measures invariant with respect to a large class 
of tree reduction operations.
Such operations include the generalized dynamical pruning (also known as hereditary reduction 
in a real tree setting) 
that eliminates descendant subtrees according to the value of an arbitrary subtree function that is monotone nondecreasing 
with respect to an isometry-induced partial tree order.
We show that, under a mild regularity condition, the IGW measures are the only attractors of 
critical Galton-Watson measures with respect to the generalized dynamical pruning.
We also derive the distributions of height, length, and size of the IGW trees.  
\end{abstract}

\footnotetext[1]{Department of Mathematics, Oregon State University, Corvallis, OR 97331-4605}
\footnotetext[2]{Department of Mathematics and Statistics, University of Nevada Reno, NV 89557-0084}

\tableofcontents

\section{Introduction}

In this paper we continue the study of the critical branching processes with the progeny generating function $Q(z)=z+q(1-z)^{1/q}$ for a given parameter $q \in [1/2,1)$.
The importance of these  processes was previously noticed in \cite{AD2015,DLG2002,Winkel2012,KZ20survey,KZ21,LeJan91,Neveu86,Zolotarev1957}.
The random tree measures induced by these branching processes are called here the {\it Invariant Galton-Watson (IGW)} measures.
This paper has two goals.  
First, we establish the main metric properties of the IGW trees:
the distributions of its height, lengths, and size (the number of edges). 
These distributions are well-studied for a special case of the IGW process with $q=1/2$ 
that coincides with 
the critical binary Galton-Watson process \cite{KZ20survey}.
Here we establish the results for the general case of $q \in [1/2,1)$. 
Second, we extend the results of \cite{KZ21}, where the IGW trees were shown to be the attractors
of the pushforward measures under the iterative application of Horton pruning 
(eliminating tree leaves followed by a series reduction). 
Here, we obtain analogous results under a much broader 
{\it generalized dynamical pruning} introduced in \cite{KZ20}.
Since the generalized dynamical pruning can be expressed via the {\it hereditary reduction} of \cite{Winkel2012}, 
the attractor property of IGW tree measures holds for the hereditary reduction as well. 
Also, the IGW trees turn out to be the attractors under the {\it Bernoulli leaf coloring}, a tree reduction studied in \cite{DW2007}.

The paper is organized as follows. Section~\ref{sec:background} contains all the necessary background. The results are stated in Section~\ref{sec:results} and proved in  
Section~\ref{sec:proofs}. 
The paper concludes with a discussion in Section~\ref{sec:discuss}.
Appendix~\ref{appndx:LIT} contains the statement of the Lagrange Inversion Theorem used in this paper.
In Appendix~\ref{appndx:Karamata}, the Karamata's theorem and its converse are stated. 
Appendix~\ref{appndx:ProofsGDPGW} contains a proof of Lemma \ref{lem:pruningGW}.

\section{Background}\label{sec:background}

\subsection{Spaces of trees}
A tree is called {\it rooted} if one of its vertices, denoted by $\rho$, is designated as the tree root.
The existence of root imposes a parent-offspring relation between each pair of adjacent vertices: 
the one closest to the root is called the {\it parent}, and the other the {\it offspring}. 
A tree is called {\it reduced} if it has no vertices of degree $2$, with the root as the only possible exception.
Let $\cT$ denote the space of finite unlabeled rooted reduced trees with no planar embedding.
The absence of planar embedding is the absence of order among the offspring of the same parent.  
The space $\cT$ includes the {\it empty tree} $\phi$ comprised of a root vertex $\rho$ and no edges.

Let $\cL$ denote the space of trees from $\cT$ equipped with edge lengths. Thus, a tree in $\cL$ is itself a metric space.
A metric tree $T\in\cL$ can be considered as a metric space with distance $d(\cdot,\cdot)$  
induced by the Lebesgue measure along the tree edges \cite{KZ20survey}.
Hence, a metric tree $T\in\cL$ can be represented as a pair $T=(S,d)$, where $S$ represents the space and $d$ is the metric defined on space $S$. 

A non-empty rooted tree is called {\it planted} if its root $\rho$ has degree one. In this case the only edge connected to the root is called the {\it stem}.
If the root $\rho$ is of degree $\geq 2$ then the tree is called {\it stemless}.
We denote by $\cT^{|}$ and $\cT^{\vee}$ the subspaces of $\cT$ consisting of planted and stemless
trees, respectively. Similarly, $\cL^{|}$ and $\cL^{\vee}$ are the subspaces of $\cL$ consisting of planted and stemless trees.
Additionally, we include the empty tree $\phi=\{\rho\}$ as an element in each of these subspaces, $\cT^{|}$, $\cT^{\vee}$, $\cL^{|}$, and $\cL^{\vee}$, defined above.

\subsection{Galton-Watson tree measures}

Consider a Galton-Watson branching process with a given progeny distribution (p.m.f.) $\{q_k\}$, $k=0,1,2,\hdots$.
More specifically, we consider a discrete time Markov process that begins with a single progenitor, 
which produces a single offspring (hence the examined trees are planted). 
At each later step, each existing population member produces $k\ge 0$ offsprings with probability $q_k$, independently from a prior history of the process;
see \cite{AN_book,Harris_book}.

A {\it Galton-Watson tree} is formed by the trajectory of the Galton-Watson branching process, with the progenitor corresponding to the tree root $\rho$.
The single offspring of the progenitor is represented in the tree by the vertex connected to the tree root by the stem \cite{KZ20survey}.
We denote by $\mathcal{GW}(\{q_k\})$ the probability measure on $\cT^{|}$ induced by the Galton-Watson process with progeny distribution $\{q_k\}$. 
Assuming $q_1<1$, the resulting tree is finite with probability one if and only if $\sum_{k=0}^\infty k q_k\leq 1$,
i.e., the Galton-Watson is either subcritical or critical.
In this paper we let $q_1=0$ so that the Galton-Watson process generates a {\it reduced} tree.

For a given probability mass function $\{q_k\}$ with $q_1=0$ and a positive real $\lambda$, consider a random tree $T$ in $\cL^|$ satisfying 
$\textsc{shape}(T) \stackrel{d}{\sim} \mathcal{GW}(\{q_k\})$, 
and such that, conditioned on $\textsc{shape}(T)$, the edge lengths are distributed as i.i.d. exponential random variables with parameter $\lambda$. 
Let $\mathcal{GW}(\{q_k\},\lambda)$ denote the distribution of so defined random tree $T$.  
Measures $\mathcal{GW}(\{q_k\})$ and $\mathcal{GW}(\{q_k\},\lambda)$ induced by critical (or subcritical) branching processes will be called 
{\it critical (or subcritical) Galton-Watson measures}.

\subsection{Invariant Galton-Watson measures}\label{sec:IGWq}

Invariant Galton-Watson measures is a single parameter family of critical Galton-Watson measures $\mathcal{GW}(\{q_k\})$ with $q_1=0$ on $\cT^|$ that we define as follows. 

\begin{Def}[{\bf Invariant Galton-Watson measures in $\cT^|$}]\label{def:IGWq0}
For a given $q \in [1/2,1)$, a critical Galton-Watson measure $\mathcal{GW}(\{q_k\})$ on $\cT^|$
is said to be the {\it invariant Galton-Watson (IGW)} measure with parameter $q$ and denoted by $\mathcal{IGW}(q)$ 
if its generating function is given by
\be\label{eqn:completeGWQz}
Q(z)=z+q(1-z)^{1/q}.
\ee
The respective branching probabilities are $q_0=q$, $q_1=0$, $q_2=(1-q)/2q$, and 
\be\label{eqn:completeGWqk}
q_k= {1-q \over k!\,q}\, \prod\limits_{i=2}^{k-1}(i-1/q) \quad \text{for } k\geq 3.
\ee
Here, if $q=1/2$, then the distribution is critical binary, i.e., 
$\mathcal{GW}(q_0\!= q_2\!=\!1/2)$.
If $q \in (1/2,1)$, the distribution is of Zipf type with
\be\label{eqn:completeGWqkZipf}
q_k={(1-q) \Gamma(k-1/q) \over q \Gamma(2-1/q) \,k!} \sim C k^{-(1+q)/q}, ~\text{ where }~C={1-q \over q\,\Gamma(2-1/q)}.
\ee
\end{Def}

\medskip
\noindent
This family of tree measures (Def.~\ref{def:IGWq0}) is also known as {\it stable Galton-Watson trees} or Galton-Watson trees with {\it stable offspring distribution} \cite{Winkel2012}.
They were previously considered in the work of V.\,M.~Zolotarev \cite{Zolotarev1957}, J.~Neveu \cite{Neveu86}, Y. Le Jan \cite{LeJan91}, T.~Duquesne and  J.-F.~Le Gall \cite{DLG2002}, R.~Abraham and J.-F.~Delmas \cite{AD2015}, T.~Duquesne and M.~Winkel \cite{Winkel2012}.
Moreover, in \cite{Neveu86}, J.~Neveu regards the generating functions \eqref{eqn:completeGWQz} to be the most important in the critical case. 

\medskip
\noindent
The definition of the invariant Galton-Watson (IGW) measure can be extended to $\cL^|$ by assigning i.i.d. exponentially distributed edge lengths.
\begin{Def}[{\bf Invariant Galton-Watson measures in $\cL^|$}]\label{def:IGWqLambda}
For a given $q \in [1/2,1)$ and $\lambda>0$, a random tree $T$ in $\cL^|$ 
is said to be the {\it exponential invariant Galton-Watson tree} 
if it satisfies the following properties: 
\begin{itemize}
  \item[(i)] $\textsc{shape}(T) \stackrel{d}{\sim} \mathcal{IGW}(q)$;  
  \item[(ii)] conditioned on $\textsc{shape}(T)$, the edge lengths are distributed as i.i.d. exponential random variables with parameter $\lambda>0$.
\end{itemize}
Such a tree is denoted by $\mathcal{IGW}(q,\lambda)$.
In other words, $T\stackrel{d}{\sim}\mathcal{GW}(\{q_k\},\lambda)$ 
with $q_k$ as in \eqref{eqn:completeGWqk}.
\end{Def}

\subsection{Invariance and attractor properties of IGW family under Horton pruning}\label{sec:HortonIGW}
Horton pruning of a tree $T$ (in $\cT$ or $\cL$) is done by removing all the leaves of $T$ (leaf vertices together with the corresponding adjacent edges)  
followed by consecutive series reduction (removing degree-two vertices by merging adjacent 
edges into one and adding up their lengths for a tree in $\cL$).
The resulting reduced tree is denoted by $\cR(T)$. 
We refer to \cite{KZ20survey} for a detailed treatment of Horton pruning.
The Horton pruning operator $\cR$ induces a contraction map on $\cT$ (or $\cL$). 
The trajectory of each tree $T$ under iterative application of $\cR$, i.e.,
\be\label{TRR0}
T\equiv\cR^0(T)\to \cR^1(T) \to\dots\to\cR^k(T)=\phi,
\ee
is uniquely determined and finite with the empty tree $\phi$ as the (only) fixed point.
In \cite{KZ20}, it was established that, under iterative Horton pruning, 
the $\mathcal{IGW}(q)$ measures are the attractors 
of all critical Galton-Watson trees that satisfy the following 
regularity assumption.
\begin{assum}\label{asm:reg}
Consider a critical Galton-Watson measure $\mathcal{GW}(\{q_k\})$ with $q_1=0$ and the respective 
progeny generating function $Q(z)$.
We assume that the following limit exists:
\be\label{eqn:RegAsmQ}
\lim\limits_{x \rightarrow 1-}{Q(x)-x \over (1-x)\big(1-Q'(x)\big)}.
\ee
\end{assum}

\medskip
\noindent
We will use function $g(x)$ defined in the following proposition from \cite{KZ20}.
\begin{prop}[{{\bf \cite{KZ21}}}]\label{prop:Qminusx}
Consider a critical Galton-Watson measure $\mathcal{GW}(\{q_k\})$ with $q_1=0$ and 
the respective progeny generating function $Q(z)$. 
Then,
$$Q(x)-x=(1-x)^2 g(x)$$
where $g(x)$ is defined as follows. Let $X \stackrel{d}{\sim} \{q_k\}$ be a progeny random variable, then
\be\label{eqn:gxDef}
g(x)=\sum\limits_{m=0}^\infty {\sf E}\big[(X-m-1)_+\big] \, x^m ~=\sum\limits_{m=0}^\infty \sum\limits_{k=m+1}^\infty (k-m-1)q_k \,x^m,
\ee
where $\,x_+=\max\{x,0\}$.
\end{prop}

\medskip
\noindent
An important limit is defined in the following lemma.
\begin{lem}[{{\bf \cite{KZ21}}}]\label{lem:dSof1}
Consider a critical Galton-Watson measure $\mathcal{GW}(\{q_k\})$ with $q_1=0$.
If Assumption \ref{asm:reg} is satisfied, then for $g(x)$ defined in \eqref{eqn:gxDef} 
the following limit exists
\be\label{eqn:gLnLimL}
\lim\limits_{x \rightarrow 1-}\left({\ln{g(x)} \over -\ln(1-x)}\right)=L,
\ee
and
$\, \lim\limits_{x \rightarrow 1-}{Q(x)-x \over (1-x)(1-Q'(x))}={1 \over 2-L}$.
\end{lem}

\medskip
\noindent
The following three results (Lem.~\ref{lem:2plusMoment}, \ref{lem:RegCond}, and \ref{lem:Zipf}) concerning the applicability of Assumption~\ref{asm:reg} and the limit $L$ in \eqref{eqn:gLnLimL} 
were established in \cite{KZ21}.
\begin{lem}[{{\bf \cite{KZ21}}}]\label{lem:2plusMoment}
Consider a critical Galton-Watson measure $\mathcal{GW}(\{q_k\})$ with $q_1=0$. 
For a progeny variable $X \stackrel{d}{\sim} \{q_k\}$ and $g(x)$ in \eqref{eqn:gxDef}, if
\be\label{eqn:2plusMoment}
{\sf E}[X^{2-\epsilon}]=\sum \limits_{k=0}^\infty k^{2-\epsilon}q_k ~<\infty \qquad \forall \epsilon>0,
\ee
then $~L=\lim\limits_{x \rightarrow 1-}\left({\ln{g(x)} \over -\ln(1-x)}\right)=0$.
If moreover the second moment is finite, i.e., $${\sf E}[X^2]=\sum \limits_{k=0}^\infty k^2q_k ~<\infty,$$
then Assumption~\ref{asm:reg} is satisfied with $\, \lim\limits_{x \rightarrow 1-}{Q(x)-x \over (1-x)(1-Q'(x))}={1 \over 2}$.
\end{lem}


\noindent
Next lemma provides a basic regularity condition for Assumption \ref{asm:reg} to hold.
\begin{lem}[{{\bf Regularity condition, \cite{KZ21}}}]\label{lem:RegCond}
Consider a critical Galton-Watson measure $\mathcal{GW}(\{q_k\})$ with $q_1=0$ and infinite second moment, 
i.e., $\sum\limits_{k=0}^\infty k^2 q_k =\infty$.
Suppose that for the progeny variable $X \stackrel{d}{\sim} \{q_k\}$ the following
limit exists:
\be\label{eqn:RegCondLim}
\Lambda=\lim\limits_{k \rightarrow \infty}{k \over E[X\,|\,X\geq k]}=\lim\limits_{k \rightarrow 
\infty}{~k\sum\limits_{m=k}^\infty \!q_m ~ \over \sum\limits_{m=k}^\infty \! mq_m}.
\ee 
Then, Assumption \ref{asm:reg} is satisfied with 
$\, \lim\limits_{x \rightarrow 1-}{Q(x)-x \over (1-x)(1-Q'(x))}=1-\Lambda \,$ 
and $\,L=2+{1 \over 1-\Lambda}$.
\end{lem}

\noindent
Next lemma follows immediately from Lemma~\ref{lem:RegCond}.
\begin{lem}[{{\bf Zipf distribution, \cite{KZ21}}}]\label{lem:Zipf}
Consider a critical Galton-Watson process $\mathcal{GW}(\{q_k\})$ with $q_1=0$ and offspring distribution $\{q_k\}$ of Zipf type:
\be\label{eqn:ZipfTail}
q_k \sim Ck^{-(\alpha+1)} \quad \text{ with }~\alpha \in (1,2]~\text{ and }~C>0.
\ee 
Then, Assumption \ref{asm:reg} is satisfied, and
\be\label{eqn:ZipfTailS1g}
L=\lim\limits_{x \rightarrow 1-}\left({\ln{g(x)} \over -\ln(1-x)}\right)=2-\alpha.
\ee 
\end{lem}

\medskip
\noindent
In Sections \ref{sec:GDP} and \ref{sec:Attractors} of this paper we consider 
generalizations of the following two theorems that were proved in \cite{KZ21}.
\begin{thm}[{{\bf Self-similarity under Horton pruning, \cite{KZ21}}}]\label{thm:completeGW}
Consider a critical or subcritical Galton-Watson measure $\mu \equiv \mathcal{GW}(\{q_k\})$ with $q_1=0$ that satisfies Assumption \ref{asm:reg}.
Then, a Galton-Watson measure $\mu$ is Horton prune-invariant (self-similar), i.e., the pushforward measure $\nu(T)=\mu \circ \cR^{-1}(T) = \mu \big(\cR^{-1}(T)\big)$ satisfies
$\,\nu\left(T\,|T\ne\phi\right)=\mu(T)$,
if and only if $\mu$ is the invariant Galton-Watson (IGW) measure $\mathcal{IGW}(q)$ with $q\in[1/2,1)$.
\end{thm}

\begin{thm}[{{\bf IGW attractors under iterative Horton pruning, \cite{KZ21}}}]\label{thm:IGWattractorHorton}
Consider a critical Galton-Watson measure $\rho_0 \equiv\mathcal{GW}(\{q_k\})$ with $q_1=0$ on $\cT^|$.
Starting with $k=0$, and for each consecutive integer,
let $\nu_k=\cR_*(\rho_k)$ denote the pushforward probability measure induced by the pruning operator, i.e.,
$\nu_k(T)=\rho_k \circ \cR^{-1}(T) = \rho_k \big(\cR^{-1}(T)\big)$,
and set $\rho_{k+1}(T)=\nu_k\left(T~|T\ne\phi\right)$.
Suppose Assumption \ref{asm:reg} is satisfied. 
Then, for any $T\in\cT^|$,
$$\lim_{k\to\infty}\rho_k(T)=\rho^*(T),$$
where $\rho^*$ denotes the invariant Galton-Watson measure 
$\mathcal{IGW}(q)$ with $q={1 \over 2-L}$ and $L$ as defined in \eqref{eqn:gLnLimL}.

Finally, if the Galton-Watson measure $\rho_0 \equiv\mathcal{GW}(\{q_k\})$ is subcritical, then 
$\rho_k(T)$ converges to a point mass measure, $\mathcal{GW}(q_0\!=\!1)$.
\end{thm}

\subsection{Generalized dynamical pruning}\label{sec:gdp}

Given a metric tree $T=(S,d) \in \cL$ and a point $x \in S$, let $\Delta_{x,T}$ be the {\it descendant tree} 
of $x$: the tree comprised of all points of $T$ descendant to $x$, including $x$. 
Then $\Delta_{x,T}$ is itself a tree in $\cL$ with the root at $x$. 

\medskip
\noindent
Let $T_1=(S_1,d_1)$ and $T_2=(S_2,d_2)$ be two metric rooted trees, 
and let $\rho_1$ denote the root of $T_1$. 
A function $f: T_1 \rightarrow T_2$ is said to be an {\it isometry} if 
${\sf Image}[f] \subseteq \Delta_{f(\rho_1),T_2}$ and for all pairs $x,y \in T_1$,
$$d_2\big(f(x),f(y)\big)=d_1(x,y).$$
We use the above defined isometry to define a {\it partial order} in the space $\cL$ as follows.  
We say that $T_1$ is {\it less than or equal to} $T_2$ and write $T_1 
\preceq T_2$ if 
 there is an isometry $f: T_1 \rightarrow T_2$. 
The relation $\preceq$ is a partial order as it satisfies the 
reflexivity, antisymmetry, and transitivity conditions. 
We say that a function $\varphi:\cL \rightarrow \mathbb{R}$ is {\it monotone nondecreasing} 
with respect to the partial order $\preceq $ if
$\varphi(T_1) \leq \varphi(T_2)$ whenever $T_1 \preceq T_2.$

\medskip
\noindent
Next, we recall the definition of the {\it generalized dynamical pruning} as stated in \cite{KZ20,KZ20survey}.
Consider a {\it monotone nondecreasing} function $\varphi:\cL \rightarrow \mathbb{R}_+$ 
with respect to the above defined partial order $\preceq$.
We define the generalized dynamical pruning operator $\S(\varphi,T):\cL\rightarrow\cL$ induced by $\varphi$ 
for any given time parameter $t\ge 0$ as 
\be\label{eqn:GenDynamPruning}
\S(\varphi,T):=\{\rho\}\cup\Big\{x \in T\setminus\rho ~:~\varphi\big(\Delta_{x,T}\big)\geq t \Big\},
\ee
where $\rho$ denotes the root of tree $T$.
Informally, the operator $\S$ cuts all subtrees $\Delta_{x,T}$ for which the value of $\varphi$
is below threshold $t$, and always keeps the tree root.

\medskip
\noindent
Below we discuss some well-studied examples of generalized dynamical pruning.
\begin{ex}[{\bf Pruning via the Horton-Strahler order}]\label{ex:H}
The Horton-Strahler order \cite{KZ20survey,Pec95,BWW00,KZ16} was initially introduced in the context of geomorphology. 
It can be defined via the operation of Horton pruning $\cR$.
The Horton-Strahler order ${\sf ord}(T)$ of a planted tree from $\cL^|$ (or $\cT^|$)
is the minimal number of prunings necessary to eliminate a tree $T$.
The Horton-Strahler order ${\sf ord}(T)$ of a stemless tree from $\cL^\vee$ (or $\cT^\vee$)
equals one plus the minimal number of prunings necessary to eliminate a tree $T$.
For a tree $T$ in either $\cT$ or $\cL$, consider
\begin{equation}\label{eqn:phi_Horton}
\varphi(T) = {\sf ord}(T)-1.
\end{equation} 

\medskip
\noindent
For $k \in \mathbb{N}$, let $\,\cR^k\,$ denote the $k$-th iteration of Horton pruning $\cR$, i.e., $\,\cR^0(T)=T$ and $\,\cR^k=  \underbrace{\cR \circ \hdots \circ \cR}_{k \text{ times}}$.
With the function $\varphi$ as in \eqref{eqn:phi_Horton}, the generalized dynamical pruning operator $\S=\cR^{\lfloor t \rfloor}$  
satisfies {\it discrete semigroup property} \cite{KZ20survey,KZ20}:
$$\cS_t\circ\cS_s=\cS_{t+s} ~\text{ for any }~t,s\in \mathbb{N}_0$$
as  $\, {\cR}^t \circ {\cR}^s ={\cR}^{t+s}$.
A recent survey of results related to invariance of a tree distribution with 
respect to Horton pruning is given in \cite{KZ20survey}.
\end{ex}

\begin{ex}[{\bf Pruning via the tree height}]\label{ex:height}
If we let the $\varphi(T)$ be the height function, i.e., for a tree $T \in\cL$, let
\begin{equation}\label{eqn:phi_hight}
\varphi(T) = \textsc{height}(T),
\end{equation}
then the generalized dynamical pruning $\S(\cdot)=\S(\varphi,\cdot)$ will
coincide with the continuous pruning (leaf-length erasure) studied in Neveu \cite{Neveu86}, 
where the invariance of critical binary Galton-Watson measures with i.i.d. exponential edge lengths with respect to this operation
was established.
In this case the operator $\S$ is known to satisfy {\it continuous semigroup property} \cite{Neveu86,Winkel2012,KZ20}:
$$\cS_t\circ\cS_s=\cS_{t+s} ~\text{ for any }~t,s\ge 0.$$ 
\end{ex}

\begin{ex}[{\bf Pruning via the tree length}]\label{ex:L}
Let the function $\varphi(T)$ equal the total lengths of $T\in\cL$:
\begin{equation}\label{eqn:phi_length}
\varphi(T) = \textsc{length}(T).
\end{equation}
The pruning operator $\S(\cdot)=\S(\varphi,\cdot)$ with the pruning function $\varphi$ as in \eqref{eqn:phi_length} coincides with 
the potential dynamics of continuum mechanics formulation of the 1-D ballistic annihilation model
$A+A \rightarrow \varnothing$ \cite{KZ20}.
Importantly, the operator $\S$ induced by the length function $\varphi$ as in \eqref{eqn:phi_length} does not satisfy the semigroup property
(discrete or continuous),
i.e., $\,\cS_t\circ\cS_s \not=\cS_{t+s}$ \cite{KZ20}.  
\end{ex}

\begin{ex}[{\bf Pruning via the number of leaves}]\label{ex:numL}
Let $\textsc{leaves}(T)$ denote the number of leaves in a tree $T$.
Then
\begin{equation}\label{eqn:phi_leaves}
\varphi(T) = \textsc{leaves}(T),
\end{equation}
is another monotone nondecreasing function.
The generalized dynamical pruning operator $\S(\cdot)=\S(\varphi,\cdot)$ induced by $\varphi$ as in \eqref{eqn:phi_leaves} 
does not satisfy the semigroup property, whether discrete or continuous.
This type of pruning naturally arises in the context of Shreve stream ordering in hydrodynamics.  
\end{ex}

\subsection{Generalized dynamical pruning as a hereditary reduction}\label{sec:DWreductions}
 
Duquesne and Winkel \cite{Winkel2012} introduced a very general kind of tree reduction in 
the context of {\it complete locally compact rooted (CLCR) real trees}, which include all the trees in $\cL$.
In \cite{Winkel2012}, a {\it hereditary property} $A$ is defined as a Borel subset in the space $\mathbb{T}$ of CLCR real trees
(more precisely, their equivalence classes under isometry) equipped with the pointed Gromov--Hausdorff metric
such that for a CLCR real tree $T \in \mathbb{T}$ and any $x\in T$,
$$\Delta_{x,T} \in A  \quad \Rightarrow \quad T=\Delta_{\rho,T} \in A.$$
As an example of a hereditary property, one may consider $A=\{T \in \mathbb{T} :\, \textsc{height}(T)\geq t\}$.

\medskip
\noindent
A hereditary property $A \subset \mathbb{T}$ induces a {\it hereditary reduction} operator $R_A:\mathbb{T} \to \mathbb{T}$ defined as
\be\label{eqn:HereditaryReduc}
R_A(T):=\{\rho\}\cup\big\{x \in T\setminus\rho ~:~\Delta_{x,T} \in A \big\}.
\ee

\medskip
\noindent
The following result was proved in \cite[Theorem 2.18]{Winkel2012}.
\begin{thm}[{{\bf Evolution of Galton-Watson trees under hereditary reduction, \cite{Winkel2012}}}]\label{thm:hrGW}
Consider a critical or subcritical Galton-Watson measure $\mu\equiv\mathcal{GW}(\{q_k\},\lambda)$  ($q_1=0$) on $\cL^|$ with generating function $Q(z)$.
For a hereditary property $A \subset \mathbb{T}$, let $\nu$ denote the corresponding pushforward probability measure induced by the hereditary reduction $R_A$,
$$\nu(T)=\mu \circ R_A^{-1}(T) = \mu \big(R_A^{-1}(T)\big).$$
Then, $\nu\big(T \in \cdot \,|R_A(T)\not= \phi\big)\stackrel{d}{=}\mathcal{GW}\left(\{g_k\},\, \lambda \big(1-Q'(1-p)\big)\right)$ is a Galton-Watson tree measure over $\cL^|$ with
independent exponential edge lengths with parameter $\lambda \big(1-Q'(1-p)\big)$,
and generating function
\be\label{eqn:hrGW}
G(z)=z+ {Q\big((1-p)+pz\big)-(1-p) -pz \over  p\big(1-Q'(1-p)\big)},
\ee
where $\,p={\sf P}\big(R_A(T)\not= \phi\big)$.
\end{thm}

\medskip
\noindent
Observe the following direct link between the operations of generalized dynamical pruning and hereditary reduction.
Consider a Borel measurable monotone nondecreasing function $\varphi:\cL \rightarrow \mathbb{R}_+$. 
Then, for a fixed $t \geq 0$, the Borel set
\be\label{eqn:gdpAtHR}
A=\{T \in \mathbb{T} :\, \varphi(T)\geq t\}
\ee
is a hereditary property, and therefore $\,\S(\varphi,T)=R_A(T)\,$ is a hereditary reduction.

\medskip
\noindent
The composition of two hereditary properties $A$ and $A'$ was defined in \cite[Def.~2.12]{Winkel2012} as the set
$$A' \circ A=\{T \in \mathbb{T} :\, R_A(T)\in A' \}.$$
Consequently, in Lemma 2.13 of \cite{Winkel2012}, the hereditary reductions were shown to satisfy the composition property, $\,R_{A' \circ A}=R_{A'} \circ R_A$.
Importantly, if we let $A_t$ denote the hereditary property in \eqref{eqn:gdpAtHR}, then
$$A_t \circ A_s \not= A_{s+t}$$
for many (or rather, all but a few) functions $\varphi$, e.g. $\varphi(T)=\textsc{length}(T)$.
Speaking of the exceptions, equation $A_t \circ A_s = A_{s+t}$ is known to hold 
for $\varphi(T)=\textsc{height}(T)$ with real $s,t \in [0,\infty)$ corresponding to Neveu (leaf-length) erasure as in Example~\ref{ex:height},
and for $\varphi(T)={\sf ord}(T)-1$ with integer $s,t \in \mathbb{Z}_+$ corresponding to Horton pruning as is Example~\ref{ex:H}.

\medskip
\noindent
We will need the following adaptation of Theorem \ref{thm:hrGW} for generalized dynamical pruning.
\begin{lem}[{\bf Pruning Galton-Watson trees}]\label{lem:pruningGW}
Consider a critical or subcritical Galton-Watson measure $\mu\equiv\mathcal{GW}(\{q_k\},\lambda)$  ($q_1=0$) on $\cL^|$ with generating function $Q(z)$.
For a monotone nondecreasing function $\varphi:\cL \rightarrow \mathbb{R}_+$, let $\nu$ denote the corresponding pushforward probability measure induced by the pruning operator $\S(T)=\S(\varphi,T)$,
$$\nu(T)=\mu \circ \S^{-1}(T) = \mu \big(\S^{-1}(T)\big).$$
Then, $\nu\big(T \in \cdot \,|T\not= \phi\big)\stackrel{d}{=}\mathcal{GW}\left(\{g_k\},\, \lambda \big(1-Q'(1-p_t)\big)\right)$ is a Galton-Watson tree measure over $\cL^|$ with
independent exponential edge lengths with parameter $\lambda \big(1-Q'(1-p_t)\big)$,
offspring probabilities
\be\label{eqn:pruningpk}
g_0={Q(1-p_t)-(1-p_t) \over p_t\big(1-Q'(1-p_t)\big)}, ~~~g_1=0, ~~\text{ and }~~
g_m={p_t^{m-1} \over m!} Q^{(m)}\!(1-p_t) \,(1-Q'(1-p_t))^{-1} \quad (m \geq 2),
\ee
where $\,p_t=p_t(\lambda,\varphi)={\sf P}\big(\cS_t(\varphi,T) \not= \phi\big)$, and generating function
\be\label{eqn:pruningG}
G(z)=z+ {Q\big((1-p_t)+p_t z\big)-(1-p_t) -p_t z \over  p_t\big(1-Q'(1-p_t)\big)}.
\ee
Moreover, if $\mu(T\in \cdot)$ is critical, then so is  $\nu\big(T \in \cdot \,\big|\,T\not= \phi\big)$.
\end{lem}

\noindent
An alternative proof of Lemma \ref{lem:pruningGW} can be found in Appendix \ref{appndx:ProofsGDPGW}.
Since Lemma \ref{lem:pruningGW} deals with the finite-leaf trees $(\textsc{leaves}(T)<\infty$), this lemma and its proof, as well as the whole set-up of generalized dynamical pruning, 
do not require introducing Gromov--Hausdorff metric and requiring 
the function $\varphi:\cL \rightarrow \mathbb{R}_+$ to be Borel measurable.

\subsection{Bernoulli leaf coloring}\label{sec:coloring}

Duquesne and Winkel considered the following type of tree reduction in \cite{DW2007}.
Fix probability $p \in [0,1)$.  
For a finite tree $T \in \cT^|$ (or $\cL^|$), select a subset of its leaves
via performing $\textsc{leaves}(T)$ independent Bernoulli trials, where each leaf is independently selected in with probability $1-p$.
Let $\mathcal{C}_p(T)$ be the minimal subtree of $T$ that contains all selected leaves and the root $\rho$.
If $T$ is a random tree, then so is $\mathcal{C}_p(T)$. 
Notice that $\mathcal{C}_p$ is a random operator induced by a countable sequence of independent Bernoulli random variables.
\begin{thm}[{{\bf Evolution of Galton-Watson trees under Bernoulli leaf coloring, \cite{DW2007}}}]\label{thm:coloring}
Consider a critical or subcritical Galton-Watson measure $\mu\equiv\mathcal{GW}(\{q_k\})$  ($q_1=0$) on $\cT^|$ with generating function $Q(z)$.
Then, for a given $p \in [0,1)$, $\,\mu\big(\mathcal{C}_p(T) \in \cdot \,|\,\mathcal{C}_p(T)\not= \phi\big)\,$ is a Galton-Watson tree measure over $\cT^|$ with
the generating function
\be\label{eqn:coloring}
G_p(z)=z+ {Q\big((1-p)+g_pz\big)-(1-g_p) -g_pz \over  g_p\big(1-Q'(1-g_p)\big)},
\ee
where $\,g_p={\sf P}\big(\mathcal{C}_p(T)\not= \phi\big)$.
\end{thm}

Theorem~\ref{thm:coloring} readily implies that the IGW trees are invariant with respect to Bernoulli leaf coloring.

\section{Results}\label{sec:results}

\subsection{Metric properties of invariant Galton-Watson trees}\label{sec:IGWproperties}

Here we derive explicit formulas for selected metric properties of $\mathcal{IGW}(q)$ and $\mathcal{IGW}(q,\lambda)$ trees in respective spaces, $\cT^|$ and $\cL^|$.
This includes the tree height distribution (Thm.~\ref{thm:heightIGW}), the tree length distribution (Thm.~\ref{thm:TreeLD}), the tree size (number of edges) distribution (Thm.~\ref{thm:edgesIGWq}) as well as the tail asymptotics for the distributions of the tree length (Prop.~\ref{prop:TreeLtail}) and tree size (Prop.~\ref{prop:edgesIGWtail}).
The proofs are collected in Sect.~\ref{sec:ProofsIGWproperties}.
\begin{thm}[{\bf Tree height distribution}]\label{thm:heightIGW}
Let $T\in\cL^|$ be an invariant Galton-Watson tree with parameters $q \in [1/2,1)$ and $\lambda>0$, i.e., $T\stackrel{d}{\sim}\mathcal{IGW}(q,\lambda)$.
Then the height of the tree $T$ has the cumulative distribution function 
$$H(x)={\sf P}\big(\textsc{height}(T) \leq x\big)=1-\big(\lambda(1-q) x+1\big)^{-q/(1-q)}, \qquad x\geq 0.$$
\end{thm}

\noindent
Notice that for the case $q=1/2$, we have  $H(x)={\lambda x \over \lambda x +2}$ which matches the result in \cite{KZ20survey}.

\bigskip
\begin{thm}[{\bf Tree length distribution}]\label{thm:TreeLD}
Let $T\in\cL^|$ be an invariant Galton-Watson tree with parameters $q \in [1/2,1)$ and $\lambda>0$, i.e., $T\stackrel{d}{\sim}\mathcal{IGW}(q,\lambda)$.
Then the length of the tree $T$ has the probability density function 
\be\label{eqn:ELLx}
\ell(x)=\sum_{n=1}^{\infty}(-1)^{n-1}\frac{\Gamma(n/q+1)}{n!\,(n-1)!\,\Gamma(n/q-n+2)}(\lambda q)^n x^{n-1}, \qquad x\geq 0,
\ee
and the cumulative distribution function
\be\label{eqn:ELLxCDF}
L(x)={\sf P}\big(\textsc{length}(T) \leq x\big)=\sum_{n=1}^{\infty}(-1)^{n-1}\frac{\Gamma(n/q+1)}{n!\,n!\,\Gamma(n/q-n+2)}(\lambda q)^n x^n, \qquad x\geq 0.
\ee
\end{thm}


\bigskip
\noindent
\begin{ex}
Let $q={1 \over 2}$. Then, $\ell(x)$ is already known (see \cite{KZ20,KZ20survey}):
\be\label{eqn:ellxI1}
\ell(x)=\frac{1}{x}e^{-\lambda x}I_1(\lambda x)=\sum_{n=0}^{\infty}\frac{\lambda^{2n+1}x^{2n}e^{-\lambda x}}{2^{2n+1} \, (n+1)! \, n!}
\ee
Next, we use the multinomial approach to show that \eqref{eqn:ellxI1} matches the equation \eqref{eqn:ELLx} for $q={1 \over 2}$.
First, we rewrite \eqref{eqn:ellxI1}:
\begin{align}\label{eqn:ellxI1ds}
\ell(x)&=e^{-\lambda x}\sum_{n=0}^{\infty}\frac{\lambda^{2n+1}}{2^{2n+1}\,(n+1)!\,n!}x^{2n}=\sum_{k=0}^{\infty}\frac{(-\lambda)^k}{k!}x^k\sum_{n=0}^{\infty}\frac{\lambda^{2n+1}}{2^{2n+1}\,(n+1)! \, n!}x^{2n} \nonumber \\
&=\sum\limits_{m=0}^\infty \left(\sum\limits_{k+2n=m}{(-1)^k 2^{-2n-1} \over k!\,(n+1)! \, n!} \right)\lambda^{m+1} x^m
=\sum\limits_{m=0}^\infty \left(\sum\limits_{k+2n=m}{(-2)^k\over k!\, (n+1)! \, n!} \right){\lambda^{m+1} x^m \over 2^{m+1}}.
\end{align}
Recall that
$$\big(z+z^{-1}+a\big)^{m+1}=\sum\limits_{n+k+j=m+1} {(m+1)! \over n!\, k! \,j!} z^n z^{-j} a^k,$$
and therefore
$${1 \over 2\pi i}\oint\limits_{|z|=1}\big(z+z^{-1}+a\big)^m \,dz=\sum\limits_{n+k+j=m+1} {(m+1)! \over n!\,k!\,j!}a^k  {1 \over 2\pi i}\oint\limits_{|z|=1} z^{n-j} \,dz
=\sum\limits_{k+2n=m} {(m+1)! \over n!\, (n+1)!\,k!}a^k,$$
implying
$$\sum\limits_{k+2n=m} {1\over n!\, (n+1)!\, k!}a^k={1 \over 2\pi i (m+1)! }\oint\limits_{|z|=1}\big(z+z^{-1}+a\big)^{m+1} \,dz.$$
Now, 
$${1 \over 2\pi i}\oint\limits_{|z|=1}\big(z+z^{-1}-2\big)^{m+1} \,dz={1 \over 2\pi i}\oint\limits_{|z|=1}{(z-1)^{2m+2} \over z^{m+1}} \,dz
={1 \over 2\pi i}\oint\limits_{|z|=1}\sum\limits_{j=0}^{2m+2}{2m+2 \choose j}(-1)^j z^{j-m-1} \,dz$$
$$=(-1)^m {2m+2 \choose m}$$
Hence,
\be\label{eqn:a2CombIdty}
\sum\limits_{k+2n=m}{(-2)^k \over k!\, (n+1)!\, n!}={1 \over (m+1)! }\,{1 \over 2\pi i}\oint\limits_{|z|=1}\big(z+z^{-1}-2\big)^{m+1} \,dz
=(-1)^m{1 \over (m+1)!}{2m+2 \choose m}
\ee

\medskip
\noindent
Thus, substituting \eqref{eqn:a2CombIdty} into \eqref{eqn:ellxI1ds}, we obtain
$$\ell(x)=\sum\limits_{m=0}^\infty (-1)^m{1 \over (m+1)!}{2m+2 \choose m}{\lambda^{m+1} x^m \over 2^{m+1}}
=\sum\limits_{m=0}^\infty (-1)^m{(2m+2)! \over (m+1)! \,m! \,(m+2)!}(\lambda q)^{m+1} x^m$$
$$=\sum\limits_{m=0}^\infty (-1)^m{\Gamma\big((m+1)/q+1\big) \over (m+1)! \,m! \, \Gamma\big((m+1)/q - m+1)}(\lambda q)^{m+1} x^m$$
for $q={1 \over 2}$, as in the equation \eqref{eqn:ELLx} of Theorem~\ref{thm:TreeLD}.
\end{ex}

\bigskip
\noindent
The following proposition is needed since computing the cumulative distribution function $L(x)$ in \eqref{eqn:ELLxCDF} 
becomes difficult (even numerically) for all values of $q \not={1 \over 2}$, i.e., $q \in (1/2,1)$.
\begin{prop}[{\bf Tail of the tree length distribution}]\label{prop:TreeLtail}
Let $T\stackrel{d}{\sim}\mathcal{IGW}(q,\lambda)$ be an invariant Galton-Watson tree in $\cL^|$ with parameters $q \in [1/2,1)$ and $\lambda>0$.
Then the cumulative distribution function $L(x)$ in \eqref{eqn:ELLxCDF} satisfies
\be\label{eqn:ELLtail}
1-L(x) \sim {1 \over (\lambda q)^q \,\Gamma(1-q)} x^{-q}.
\ee
\end{prop}

\begin{ex}
For $q={1 \over 2}$, $L(x)$ is expressed as follows  \cite{KZ20,KZ20survey}:
$$L(x)=1-e^{-\lambda x}\big(I_0(\lambda x)+I_1(\lambda x)\big).$$
Thus, since $I_a(z) \sim {1 \over \sqrt{2\pi z}}e^z$ for all $a \geq 0$, we have 
$$1-L(x)=e^{-\lambda x}\big(I_0(\lambda x)+I_1(\lambda x)\big) \sim \sqrt{2 \over \lambda \pi}x^{-1/2}={1 \over (\lambda q)^q \,\Gamma(1-q)}x^{-q}
\quad\text{ for }\,q={1 \over 2}$$
as $\,\Gamma(1/2)=\sqrt{\pi}$. 
This matches the general case in Prop.~\ref{prop:TreeLtail}. 
\end{ex}

\bigskip
\noindent
The following is a discrete analog of Theorem~\ref{thm:TreeLD}.
\begin{thm}[{\bf Tree size distribution}]\label{thm:edgesIGWq}
Let $T\in\cT^|$ be an invariant Galton-Watson tree with parameters $q \in [1/2,1)$, i.e., $T\stackrel{d}{\sim}\mathcal{IGW}(q)$.
Then, the number of edges in $T$ is distributed with the probability mass function
\be\label{eqn:edgesIGWq}
\alpha(n)=\sum\limits_{k=1}^n (-1)^{k-1} {n-1 \choose k-1}{\Gamma(k/q+1) \over k!\,\Gamma(k/q -k+2)}\,q^k \qquad \text{ for }~n=1,2,\hdots,
\ee
with the cumulative distribution function
\be\label{eqn:CDFedgesIGWq}
\mathcal{A}(x)=\sum\limits_{k=1}^{\lfloor x \rfloor} (-1)^{k-1} {\lfloor x \rfloor \choose k}{\Gamma(k/q+1) \over k!\,\Gamma(k/q -k+2)}\,q^k, \qquad x \geq 1.
\ee
\end{thm}

\bigskip
\noindent
Next proposition is analogous to Prop.~\ref{prop:TreeLtail} and has a similar proof. It gives an estimate on the tail distribution $1-\mathcal{A}(x)$. 
\begin{prop}[{\bf Tail of the tree size distribution}]\label{prop:edgesIGWtail}
Let $T\stackrel{d}{\sim}\mathcal{IGW}(q)$ be an invariant Galton-Watson tree in $\cT^|$ with parameters $q \in [1/2,1)$.
Then the cumulative distribution function $\mathcal{A}(x)$ in \eqref{eqn:CDFedgesIGWq} satisfies
\be\label{eqn:edgesIGWtail}
1-\mathcal{A}(x) \sim {1 \over q^q \,\Gamma(1-q)} x^{-q}.
\ee
\end{prop}

\subsection{Invariance under generalized dynamical pruning}\label{sec:GDP}
Here we consider invariance (Prop.~\ref{prop:GDPinvarianceIGW}) and uniqueness (Lem.~\ref{lem:onlyIGW}) properties of $\mathcal{IGW}(q,\lambda)$ measures under generalized dynamical prunings.
Although both Prop.~\ref{prop:GDPinvarianceIGW} and Lem.~\ref{lem:onlyIGW} follow immediately from the results of Duquesne and Winkel \cite[Sect.~3.2.1]{Winkel2012}, 
alternative proofs of these statements that do not rely on a real tree setting are presented in Sect.~\ref{sec:ProofsGDP}.

\medskip
\noindent
We say that a Galton-Watson tree measure $\mu$ is invariant under the operation of pruning $\S(\cdot)=\S(\varphi,\cdot)$ if
for $\,T \stackrel{d}{\sim} \mu$,
$${\sf P}\big(\textsc{shape}(\S(T))=\tau \,\big|\,\S(T)\not=\phi \big)=\mu\big(\textsc{shape}(T)=\tau\big), \qquad \text{ for all }\, \tau \in \cT^|.$$

\begin{prop}[{\bf Invariance with respect to generalized dynamical pruning}]\label{prop:GDPinvarianceIGW}
Let $T\stackrel{d}{\sim}\mathcal{IGW}(q,\lambda)$ be an invariant Galton-Watson tree 
with parameters $q \in [1/2,1)$ and $\lambda>0$.
Then, for any monotone nondecreasing function $\varphi:\cL^|\rightarrow\mathbb{R}_+$ and
any $t>0$ we have
\[T^t:=\{\cS_t(\varphi,T)|\cS_t(\varphi,T) \not= \phi\} \stackrel{d}{\sim} 
\mathcal{IGW}\left(q,\,\cE _t(\lambda)\right),\]
where $\cE _t(\lambda)=\lambda p_t^{(1-q)/q}$ and  $p_t=p_t(\lambda,\varphi)={\sf P}(\cS_t(\varphi,T) \not= \phi)$. 
\end{prop}
\noindent
In other words, Prop.~\ref{prop:GDPinvarianceIGW} yields the invariance of $\mathcal{IGW}(q,\lambda)$ measure under generalized dynamical prunings $\S$.
For $\varphi(T) = {\sf ord}(T)-1$, Prop.~\ref{prop:GDPinvarianceIGW} yields the `if' part of Thm.~\ref{thm:completeGW}.
Next, we formulate the following uniqueness result.
\begin{lem}[{\bf Uniqueness of IGW measures}]\label{lem:onlyIGW} 
Consider a critical Galton-Watson tree measure $\mu\equiv\mathcal{GW}(\{q_k\},\lambda)$  ($q_1=0$) on $\cL^|$, and
let $T\stackrel{d}{\sim} \mu$.
Let $\varphi:\cL \rightarrow \mathbb{R}_+$ be a monotone non-decreasing function such that $\,p_t={\sf P}(\cS_t(\varphi,T) \not= \phi)\,$ is 
a decreasing function of $t$, mapping $[0,\infty)$ onto $(0,1]$. 
Then, $\mu$ is invariant under the operation of pruning $\S(T)=\S(\varphi,T)$ 
if and only if $\mu\equiv \mathcal{IGW}(q_0,\lambda)$.
\end{lem}
\noindent
Notice that Lem.~\ref{lem:onlyIGW} does not imply the uniqueness result in Thm.~\ref{thm:completeGW}, which is valid under the regularity Asm.~\ref{asm:reg}.
Next, we list some examples where the assumptions of Lem.~\ref{lem:onlyIGW} are satisfied.
\begin{ex}
Let $\varphi(T)=\textsc{height}(T)$. Consider a critical Galton-Watson tree measure $\mu\equiv\mathcal{GW}(\{q_k\},\lambda)$  ($q_1=0$) on $\cL^|$, and
let $T\stackrel{d}{\sim} \mu$.
Then, $1-p_t=P_{1,0}(t)$ is the probability of extinction by time $t$ of the critical continuous time branching process. 
Since $P_{1,0}(t)$ is a continuous function of $t$, mapping $[0,\infty)$ onto $[0,1)$, Lemma \ref{lem:onlyIGW} implies 
$IGW(q,\lambda)$ is the only class of Galton-Watson measures that are invariant under the generalized dynamical pruning with $\varphi(T)=\textsc{height}(T)$.
\end{ex}

\begin{ex}
Let $\varphi(T)=\textsc{length}(T)$. Consider a critical Galton-Watson tree measure $\mu\equiv\mathcal{GW}(\{q_k\},\lambda)$  ($q_1=0$) on $\cL^|$, and
let $T\stackrel{d}{\sim} \mu$. Denote by $N$ the number of edges in $T$. Then, the density function of $\textsc{length}(T)$ can be expressed as
$\,\sum\limits_{k=1}^{\infty}P(N=k)f_{k,\lambda}(x)$,
where $f_{k,\lambda}(x)$ is a Gamma function 
$\,f_{k,\lambda}(x)={\lambda^k \over \Gamma(k)}x^{k-1}e^{-\lambda x}$.
Hence, the cumulative distribution function of $\textsc{length}(T)$, 
$${\sf P}(\textsc{length}(T) \leq t)=1-p_t,$$
 is a continuous function of $t$, mapping $[0,\infty)$ onto $[0,1)$.
Thus, by Lemma \ref{lem:onlyIGW}, 
$IGW(q,\lambda)$ is the only class of Galton-Watson measures invariant under the generalized dynamical pruning with $\varphi(T)=\textsc{height}(T)$.
\end{ex}

\medskip
\noindent
Next, we check that Proposition~\ref{prop:GDPinvarianceIGW} and Theorem~\ref{thm:heightIGW} are consistent with this semigroup property
of the generalized dynamical pruning induced by $\varphi(T)=\textsc{height}(T)$ as in Example~\ref{ex:height}.
Indeed, for $T\stackrel{d}{\sim}\mathcal{IGW}(q,\lambda)$, Prop.~\ref{prop:GDPinvarianceIGW} yields
$$T^t:=\{\cS_t(\varphi,T)|\cS_t(\varphi,T) \not= \phi\} \stackrel{d}{\sim} \mathcal{IGW}\left(q,\,\cE _t(\lambda)\right),$$
where by Thm.~\ref{thm:heightIGW}, $\cE _t(\lambda)=\lambda p_t^{(1-q)/q}={\lambda \over \lambda(1-q) t+1}$.
Hence,
$$\cE_s \circ \cE_t (\lambda)=\cE_s \big(\cE_t (\lambda)\big)=\cE_{t+s} (\lambda),$$ 
thus reaffirming the semigroup property of $\cS_t$ for $\varphi(T)=\textsc{height}(T)$.

\subsection{Invariant Galton-Watson trees $\mathcal{IGW}(q)$ as attractors}\label{sec:Attractors}
The following result extends Theorem \ref{thm:IGWattractorHorton} to all generalized dynamical pruning operators $\S(T)=\S(\varphi,T)$.
\begin{thm}[{{\bf IGW attractors under generalized dynamical pruning}}]\label{thm:IGWattractor}
Consider a Galton-Watson measure $\mu \equiv\mathcal{GW}(\{q_k\}, \lambda)$ with $q_1=0$ on $\cL^|$.
Suppose the measure is critical and Assumption \ref{asm:reg} is satisfied.  
Then, for any random tree $T\in\cL^|$ distributed according to $\mu$, i.e., $\,T \stackrel{d}{\sim} \mu$,
$$\lim_{t\to\infty}{\sf P}\big(\textsc{shape}(\S(T))=\tau \,\big|\,\S(T)\not=\phi \big)=\mu^*(\tau), \qquad \text{ for all }\, \tau \in \cT^|,$$
where $\mu^*$ denotes the invariant Galton-Watson measure $\mathcal{IGW}(q)$ with $q={1 \over 2-L}$ and $L$ defined in \eqref{eqn:gLnLimL}.

Finally, suppose the Galton-Watson measure $\mu \equiv\mathcal{GW}(\{q_k\},\lambda)$ (with $q_1=0$) is subcritical, then 
for $\,T \stackrel{d}{\sim} \mu$, the distribution 
${\sf P}\big(\textsc{shape}(\S(T))=\cdot \,\big|\,\S(T)\not=\phi \big)$ converges to a point mass measure, $\mathcal{GW}(q_0\!=\!1)$.
\end{thm}
\noindent
Theorem \ref{thm:IGWattractor} is proved in Section \ref{sec:ProofsAttractors}.

\medskip
\noindent
Next two corollaries of Theorem~\ref{thm:IGWattractor} follow immediately from Lemmas~\ref{lem:2plusMoment} and \ref{lem:Zipf}.
\begin{cor}[{{\bf Attraction property of critical Galton-Watson trees of Zipf type}}]\label{cor:SZipfIGWattractor}
Consider a critical Galton-Watson process $\mu \equiv \mathcal{GW}(\{q_k\})$ with $q_1=0$, with offspring distribution $q_k$ of Zipf type, 
i.e., $q_k \sim C k^{-(\alpha+1)}$, with $\alpha \in (1,2]$ and $C>0$.
Then, for any a random tree $T\in\cL^|$ distributed according to $\mu$, i.e., $\,T \stackrel{d}{\sim} \mu$,
$$\lim_{t\to\infty}{\sf P}\big(\textsc{shape}(\S(T))=\tau \,\big|\,\S(T)\not=\phi \big)=\mu^*(\tau), \qquad \text{ for all }\, \tau \in \cT^|,$$
where $\mu^*$ is the invariant Galton-Watson measure $\mathcal{IGW}\left({1 \over \alpha}\right)$.
\end{cor}

\begin{cor}[{{\bf Attraction property of critical binary Galton-Watson tree, \cite{BWW00}}}]\label{BWW00_1}
Consider a critical Galton-Watson process $\mu \equiv \mathcal{GW}(\{q_k\})$ with $q_1=0$.
Assume one of the following two conditions holds.
\begin{description}
  \item[(a)] The second moment assumption is satisfied:
  $$\sum \limits_{k=2}^\infty k^2 q_k ~<\infty.$$
  \item[(b)] Assumption \ref{asm:reg} is satisfied, and the ``$2-$'' moment assumption is satisfied, i.e.,
  $$\sum \limits_{k=2}^\infty k^{2-\epsilon}q_k ~<\infty \qquad \forall \epsilon>0.$$
\end{description}
Then, for any a random tree $T\in\cL^|$ distributed according to $\mu$, i.e., $\,T \stackrel{d}{\sim} \mu$,
$$\lim_{t\to\infty}{\sf P}\big(\textsc{shape}(\S(T))=\tau \,\big|\,\S(T)\not=\phi \big)=\mu^*(\tau), \qquad \text{ for all }\, \tau \in \cT^|,$$
where $\mu^*$ is the critical binary Galton-Watson measure $\mathcal{IGW}(1/2)$.
\end{cor}

\medskip
\noindent
Next, we state a result for Bernoulli leaf coloring operator $\mathcal{C}_p$ (see Sect.~\ref{sec:coloring}), 
analogous to the one in Theorem~\ref{thm:IGWattractor}.
\begin{thm}[{{\bf IGW attractors under Bernoulli leaf coloring}}]\label{thm:coloringIGWattractor}
Consider a Galton-Watson measure $\mu \equiv\mathcal{GW}(\{q_k\})$ with $q_1=0$ on $\cT^|$.
Suppose the measure is critical and Assumption \ref{asm:reg} is satisfied.  
Then, for any a random tree $T\in\cT^|$ distributed according to $\mu$, i.e., $\,T \stackrel{d}{\sim} \mu$,
$$\lim_{p \to 1-}{\sf P}\big(\mathcal{C}_p(T)=\tau \,|\,\mathcal{C}_p(T)\not= \phi\big)=\mu^*(\tau), \qquad \text{ for all }\, \tau \in \cT^|,$$
where $\mu^*$ denotes the invariant Galton-Watson measure $\mathcal{IGW}(q)$ with $q={1 \over 2-L}$ and $L$ as defined in \eqref{eqn:gLnLimL}.

Suppose $\mu \equiv\mathcal{GW}(\{q_k\})$ (with $q_1=0$) is subcritical, then for $\,T \stackrel{d}{\sim} \mu$, the conditional distribution 
${\sf P}\big(\mathcal{C}_p(T)=\cdot \,|\,\mathcal{C}_p(T)\not= \phi\big)$ converges to a point mass measure, $\mathcal{GW}(q_0\!=\!1)$.
\end{thm}
\noindent
Theorem \ref{thm:coloringIGWattractor} is proved in Section \ref{sec:ProofsAttractors}.

\medskip
\noindent
Finally, another result analogous to Theorem~\ref{thm:IGWattractor} can be obtained for iterative hereditary reductions (see Sect.~\ref{sec:DWreductions}).
\begin{thm}[{{\bf IGW attractors under generalized hereditary reductions}}]\label{thm:hrIGWattractor}
Consider a Galton-Watson measure $\mu \equiv\mathcal{GW}(\{q_k\}, \lambda)$ with $q_1=0$ on $\cL^|$.
Suppose the measure is critical and Assumption \ref{asm:reg} is satisfied. 
Let $T\in\cL^|$ be a random tree distributed according to $\mu$, and
let $H_1,\,H_2,\hdots$ be a sequence of hereditary properties satisfying
$$\lim\limits_{n \to \infty}{\sf P}\big(R_{H_n}\circ \hdots \circ R_{H_1}(T) \not= \phi \big)=0,$$
where $R_{H_1},\,R_{H_2},\hdots$ are the corresponding hereditary properties.
Then, for $\,T \stackrel{d}{\sim} \mu$,
$$\lim_{t\to\infty}{\sf P}\big(\textsc{shape}(\S(T))=\tau \,\big|\,\S(T)\not=\phi \big)=\mu^*(\tau), \qquad \text{ for all }\, \tau \in \cT^|,$$
where $\mu^*$ denotes the invariant Galton-Watson measure $\mathcal{IGW}(q)$ with $q={1 \over 2-L}$ and $L$ as defined in \eqref{eqn:gLnLimL}.

If $\mu$ is a subcritical Galton-Watson measure, then for $\,T \stackrel{d}{\sim} \mu$, the conditional distribution 
${\sf P}\big(\textsc{shape}(\S(T))=\cdot \,\big|\,\S(T)\not=\phi \big)$ converges to a point mass measure, $\mathcal{GW}(q_0\!=\!1)$.
\end{thm}
\noindent
Theorem \ref{thm:hrIGWattractor} is proved in Section \ref{sec:ProofsAttractors}

\section{Proofs}\label{sec:proofs}

\subsection{Metric properties of invariant Galton-Watson trees}\label{sec:ProofsIGWproperties}

\begin{proof}[Proof of Theorem~\ref{thm:heightIGW}]
Consider a tree $T\stackrel{d}{\sim}\mathcal{IGW}(q,\lambda)$. Let $X$ denote the length of the stem connecting the random tree's root $\rho$ to the root's only child vertex $v_0$.
Let $K={\sf br}(v_0)$ be the branching number of $v_0$, and let the $K$ subtrees branching out of $v_0$ be denoted by $T_i, ~1 \le i \le K$. 
Let $H(x)$ be the cumulative distribution function for the height of $T$. Then, for each subtree $T_i\stackrel{d}{\sim}\mathcal{IGW}(q,\lambda)$, 
its height $\textsc{height}(T_i)$ has the same cumulative distribution function $H(x)$. The number of subtrees $K\stackrel{d}{\sim} q_k$ has generating function $Q(z)=z+q(1-z)^{1/q}$. 
Let $M(x)$ denote the cumulative distribution function of 
$\,\max\limits_{1 \le i \le K}\{\textsc{height}(T_i)\}$, then 
\begin{align}\label{eqn:QoH}
M(x)&=P\big(\max\limits_{1 \le i \le K}\{\textsc{height}(T_i)\} \leq x\big)=\sum\limits_{k=0}^\infty q_k P\big(\max\limits_{1 \le i \le K}\{\textsc{height}(T_i)\} \leq x \,\big|\, K=k\big) \nonumber \\
&=\sum\limits_{k=0}^\infty q_k P\big(\textsc{height}(T) \leq x\big)^k=\sum\limits_{k=0}^\infty q_k \big(H(x)\big)^k \nonumber \\
&=(Q\circ H)(x)=H(x)+q\big(1-H(x) \big)^{1/q}.
\end{align}
The stem length $X$ is an exponentially distributed random variable with parameter $\lambda$, and density function $\varphi_\lambda(x)=\lambda \exp\{-\lambda x\} {\bf 1}_{x \geq 0}$. 
Since, $\,\textsc{height}(T)=X+\max\limits_{1 \le i \le K}\{\textsc{height}(T_i)\}$, we have
\be\label{eqn:convPhiM}
H(x) = \varphi_\lambda \ast M(x).
\ee
We will use the following notations: let $\widehat{g}(t)=\int\limits_{-\infty}^\infty e^{itx} g(x) \,dx$ denote the Fourier transform of $g(x)$. 
Equations \eqref{eqn:QoH} and \eqref{eqn:convPhiM} yield
$$H(x)=\varphi_\lambda \ast (Q\circ H)(x).$$
Taking Fourier transform, we obtain
$$\widehat{H}(t)={\lambda \over \lambda-it}\, \left(\widehat{H}(t)+q\widehat{\big(1-H \big)^{1/q}}(t)\right),$$
which simplifies as
$$it\widehat{H}(t)+\lambda q \widehat{\big(1-H \big)^{1/q}}(t)=0,$$
where 
$$\widehat{\big(1-H \big)^{1/q}}(t)=\int\limits_{-\infty}^\infty e^{itx} \big(1-H(x) \big)^{1/q} \,dx.$$
Therefore,
\be\label{eqn:ieqnHhat}
\int\limits_{-\infty}^\infty e^{itx} \left(itH(x)+\lambda q\big(1-H(x) \big)^{1/q}\right) \,dx=0 \qquad \forall t \in \mathbb{R},
\ee
where integration by parts yields
\be\label{eqn:ipartH}
\int\limits_{-\infty}^\infty e^{itx}it H(x)\,dx=-\int\limits_{-\infty}^\infty e^{itx}H'(x)\,dx.
\ee
Substituting \eqref{eqn:ipartH} back into \eqref{eqn:ieqnHhat} yields
$$\int\limits_{-\infty}^\infty e^{itx} \left(H'(x)-\lambda q\big(1-H(x) \big)^{1/q}\right) \,dx=0 \qquad \forall t \in \mathbb{R},$$
which, by Parseval's equation implies the following ODE
\begin{equation}\label{eqn:fODE}
H'(x)=\lambda q\big(1-H(x) \big)^{1/q}.
\end{equation}

\medskip
\noindent
Next, differential equation \eqref{eqn:fODE} above via integration, obtaining
\begin{equation}\label{eqn:solODE}
H(x)=1-\big((\lambda x+C)(1-q)\big)^{-{q \over 1-q}},
\end{equation}
where $C$ is a scalar.
Since $H(x)$ is a cumulative distribution function of a positive random variable $\textsc{height}(T)$, we have $H(0)=0$, implying $C={1 \over 1-q}$.
Thus, for $q \in \left[{1 \over 2},1\right)$,
$$H(x)=1-\left(\left(\lambda x+\frac{1}{1-q}\right)(1-q)\right)^{-{q \over 1-q}}=1-\big(\lambda(1-q) x+1\big)^{-q/(1-q)}.$$
\end{proof}

\bigskip
\noindent
Next, we use the following application of the Lagrange Inversion Theorem (Thm.~\ref{thm:LIT}).
\begin{lem}\label{lem:LITapp}
Let $q \in [1/2,1)$ be given.
Suppose $W=W(z)$ is an analytic function satisfying equation 
$$z={W \over (1-W)^{1/q}}$$
in a neighborhood of the origin, where we take $-\pi< \arg(z)<\pi$ branch of the function $z^{1/q}$. 
Then, for $z$ near the origin, we have
\be\label{eqn:LITapp}
W=\sum_{n=1}^{\infty}(-1)^{n-1}\frac{\Gamma(n/q+1)}{n!\,\Gamma(n/q-n+2)}z^n.
\ee
\end{lem}
\noindent
Observe that the conclusion of Lemma~\ref{lem:LITapp} also applies in a real-valued setting, under the assumption 
of infinite differentiability of $W:\mathbb{R} \to \mathbb{R}$. Here, if $z={W \over (1-W)^{1/q}}$ for $z \in \mathbb{R}$
in a neighborhood of the origin on the real line, then the power series expansion \eqref{eqn:LITapp} holds in proximity to $0$.
\begin{proof}[Proof of Lemma~\ref{lem:LITapp}]
We notice that function $f(w)={w \over (1-w)^{1/q}}$ is analytic at $w=0$, and $f'(0)=1 ~\not=0$.
Thus, we can apply the Lagrange Inversion Theorem (Thm.~\ref{thm:LIT}) to express $\Lambda$ in terms of $z$ power series.  
Now, since
$$\left({w \over f(w)}\right)^n=(1-w)^{n/q},$$
we have
$$\frac{d^{n-1}}{dw^{n-1}}\left({w \over f(w)}\right)^{\!\!n}\!\Bigg|_{\Lambda=0}=(-1)^{n-1}(n/q)(n/q-1)\hdots(n/q-n+2)
=(-1)^{n-1}\frac{\Gamma(n/q+1)}{\Gamma(n/q-n+2)}$$
Therefore, by the Lagrange Inversion Theorem (Thm.~\ref{thm:LIT}), we obtain
$$W=\sum_{n=1}^{\infty}{z^n \over n!} \left[{d^{n-1} \over dw^{n-1}}\left({w \over f(w)}\right)^{\!\!n}\right]_{w=0}
=\sum_{n=1}^{\infty}(-1)^{n-1}\frac{\Gamma(n/q+1)}{n!\,\Gamma(n/q-n+2)}z^n.$$
\end{proof}

\bigskip
\begin{proof}[Proof of Theorem~\ref{thm:TreeLD}] 
Consider a tree $T\stackrel{d}{\sim}\mathcal{IGW}(q,\lambda)$ consisting of a stem of length $X$ that connects the root $\rho$ to its child vertex $v_0$, and $K={\sf br}(v_0)$ subtrees $T_i,~1 \le i \le K$ branching out from $v_0$. Let $\ell(x)$ be the density function of length of $T$. Notice that the length of each subtree $T_i$ is also $\ell(x)$ distributed. Random variable $K\stackrel{d}{\sim} q_k$ has generating function $Q(z)=z+q(1-z)^{1/q}$. 
Letting $N(x)$ denote the probability density function of $\,\sum\limits_{1 \le i \le K}\{\textsc{length}(T_i)\}$, we have 
\be\label{eqn:Nell}
N(x)=\sum_{k=0}^\infty q_k \ell_k(x), \quad\text{ where }~\ell_k(x)=\underbrace{\ell \ast  \hdots \ast \ell}_{k \text{ times}}(x).
\ee
Observe that $\,\textsc{length}(T)=X+\sum\limits_{1 \le i \le K}\{\textsc{length}(T_i)\}$, where
$X$ has exponential p.d.f. $\varphi_\lambda(x)=\lambda \exp\{-\lambda x\} {\bf 1}_{x \geq 0}$. 
Thus, $\ell(x)$ can be represent as the following convolution
\be\label{eqn:convPhiN}
\ell(x) = \varphi_\lambda \ast N(x).
\ee
Let for $t \geq 0$, function $\cL[g](t)=\int\limits_0^\infty e^{-tx} g(x) \,dx$  denote the Laplace transform $g$.
Then, \eqref{eqn:Nell} and \eqref{eqn:convPhiN} imply
$$\cL[\ell](t)=\cL[\varphi_\lambda](t) \,\cL[N](t)=\cL[\varphi_\lambda](t) \,Q\big(\cL[\ell](t)\big)={\lambda \over \lambda+t} \left(\cL[\ell](t)+q\big(1-\cL[\ell](t) \big)^{1/q}\right),$$
which simplifies as
\be\label{eqn:Lellqt}
t\cL[\ell](t)=\lambda q\big(1-\cL[\ell](t) \big)^{1/q}.
\ee

\bigskip
\noindent
Letting $\,z={\lambda q \over t}$ and $\Lambda=\cL[\ell]\left({\lambda q \over z}\right)=\cL[\ell](t)$, we have
$$z={\Lambda \over (1-\Lambda)^{1/q}}.$$
Then, Lemma~\ref{lem:LITapp} yields
$$\cL[\ell](t)=\Lambda=\sum_{n=1}^{\infty}(-1)^{n-1}{\Gamma(n/q+1) \over n!\, \Gamma(n/q-n+2)}\frac{(\lambda q)^n}{t^n}$$
Finally, we invert the Laplace transform $\cL[\ell](t)$, obtaining
$$\ell(x)=\sum_{n=1}^{\infty}(-1)^{n-1} {\Gamma(\alpha n+1) \over n! \,(n-1)! \,\Gamma(\alpha n-n+2)}(\lambda q)^n x^{n-1}.$$
\end{proof}

\begin{proof}[Proof of Proposition~\ref{prop:TreeLtail}]
Observe that
\begin{align*}
1-\cL[\ell](t) &= \int\limits_0^\infty (1-e^{-tx})\,\ell(x) \,dx ~= t \int\limits_0^\infty  \int\limits_0^x e^{-ty} \,\ell(x) \,dy\,dx \\
&= t \int\limits_0^\infty  e^{-ty} \, \int\limits_y^\infty \ell(x) \,dx\,dy ~= t \int\limits_0^\infty  e^{-ty} \, \big(1-L(y)\big)\,dy
~=t\,\cL[1\!-\!L](t).
\end{align*}
Thus, by \eqref{eqn:Lellqt}, we have
$$t\cL[\ell](t)=\lambda q\big(1-\cL[\ell](t) \big)^{1/q}=\lambda q \,t^{1/q}\,\big(\cL[1\!-\!L](t) \big)^{1/q},$$
and therefore,
$$\cL[1\!-\!L](t)={1 \over t^{1-q}} {\big(\cL[\ell](t)\big)^q \over (\lambda q)^q}, 
\quad \text{ where }\quad \lim\limits_{t\to 0+}{\big(\cL[\ell](t)\big)^q \over (\lambda q)^q}={1 \over (\lambda q)^q}.$$
Hence, by the Hardy-Littlewood-Karamata Tauberian Theorem for Laplace transforms \cite{FellerII},
$$1-L(x) \sim {1 \over (\lambda q)^q \,\Gamma(1-q)}x^{-q}.$$
\end{proof}

\bigskip
\begin{proof}[Proof of Theorem~\ref{thm:edgesIGWq}] 
Observe that in a reduced tree $T \in \cT^| \setminus \{\phi\}$, the number of edges equals one plus the number of edges in all subtrees 
splitting from the stem. Therefore, 
$$\alpha(n+1)=\sum_{k=n}^\infty q_k\, \underbrace{\alpha \ast  \hdots \ast \alpha}_{k \text{ times}}(n), \qquad n=0,1,\hdots.$$
Therefore, the generating function $\,a(z)=\sum\limits_{n=1}^\infty z^n \,\alpha(n)\,$ satisfies $\,a(z)=z\,Q\big(a(z)\big)$.
Hence,
\be\label{eqn:vzq}
a(z)=z\,\Big(a(z)+q\big(1-a(z)\big)^{1/q}\Big)
\ee
and therefore,
$$w={a \over (1-a)^{1/q}}, \qquad \text{ where }~a=a(z) ~~~\text{ and }~~w={q z \over 1-z}.$$

\medskip
\noindent
Lemma~\ref{lem:LITapp} yields 
\begin{align*}
a(z)&=\sum_{k=1}^{\infty} \sum\limits_{n=k}^\infty (-1)^{k-1} {n-1 \choose k-1}{\Gamma(k/q+1)\over k!\,\Gamma(k/q -k+2)}\,q^k z^n \\
&=\sum\limits_{n=1}^\infty z^n \sum_{k=1}^n  (-1)^{k-1} {n-1 \choose k-1}{\Gamma(k/q+1)\over k!\,\Gamma(k/q -k+2)}\,q^k.
\end{align*}
Thus, since $\,a(z)=\sum\limits_{n=1}^\infty  z^n \,\alpha(n)$, equation \eqref{eqn:edgesIGWq} follows.
Finally, the cumulative distribution function equals
\begin{align*}
\mathcal{A}(x)&=\sum\limits_{n=1}^{\lfloor x \rfloor}  \alpha(n) ~=\sum\limits_{n=1}^{\lfloor x \rfloor} \sum_{k=1}^n  (-1)^{k-1} {n-1 \choose k-1}{\Gamma(k/q+1)\over k!\,\Gamma(k/q -k+2)}\,q^k\\
&=\sum_{k=1}^{\lfloor x \rfloor}(-1)^{k-1}  \left(\sum\limits_{n=k}^{\lfloor x \rfloor}  {n-1 \choose k-1}\right){\Gamma(k/q+1)\over k!\,\Gamma(k/q -k+2)}\,q^k\\
&=\sum\limits_{k=1}^{\lfloor x \rfloor} (-1)^{k-1} {\lfloor x \rfloor \choose k}{\Gamma(k/q+1) \over k!\,\Gamma(k/q -k+2)}\,q^k
\end{align*}
for all real $\,x \geq 1$, and \eqref{eqn:CDFedgesIGWq} holds.
\end{proof}

\subsection{Invariance under generalized dynamical pruning}\label{sec:ProofsGDP}

\begin{proof}[Proof of Proposition~\ref{prop:GDPinvarianceIGW}]
For $q \in [1/2,1)$ and $Q(z)=z+q(1-z)^{1/q}$, equation \eqref{eqn:pruningG} in Lemma~\ref{lem:pruningGW} implies
\begin{align*}
G(z)&=z+{Q\big(1-p_t+p_t z\big)-(1-p_t)-zp_t \over p_t\big(1-Q'(1-p_t)\big)}\\
&=z+p_t^{-1/q}\left(Q\big(z+(1-z)(1-p_t)\big)-(1-p_t) -zp_t \right)\\
&=z+p_t^{-1/q}q p_t^{1/q}(1-z)^{1/q} ~=Q(z).
\end{align*}
The rest of the proof follows from Lemma~\ref{lem:pruningGW} as 
\be\label{eqn:pruningLambda}
\lambda \big(1-Q'(1-p_t)\big)=\lambda p_t^{(1-q)/q}.
\ee
yielding $\, \S(T) \stackrel{d}{\sim} \mathcal{IGW}\big(q,\lambda p_t^{(1-q)/q}\big)$.
\end{proof}

\begin{proof}[Proof of Lemma~\ref{lem:onlyIGW}]
From Lemma~\ref{lem:pruningGW}, we have
$$q_0^{(1)}=\frac{Q(p_t)-p_t}{(1-p_t)\left(1-Q'(p_t)\right)}\quad \text{ and }\quad
G(z)=z+{Q\big(p_t+(1-p_t)z\big)-p_t-z(1-p_t) \over (1-p_t)\big(1-Q'(p_t)\big)}.$$
Combining the above together yields 
$$G(z)=z+q_0^{(1)}{Q(p_t+(1-p_t)z)-(p_t+(1-p_t)z) \over Q(p_t)-p_t}$$
Suppose $\mu$ is invariant under the operation of pruning $\S(T)=\S(\varphi,T)$, then $G(z)=Q(z)$ and $q_0^{(1)}=q_0$,
implying
\be\label{eqn:QQQ}
Q(z)=z+q_0{Q(p_t+(1-p_t)z)-(p_t+(1-p_t)z) \over Q(p_t)-p_t}.
\ee
Let $R(z)=\frac{Q(z)-z}{q_0}$, then equation \eqref{eqn:QQQ} rewrites as $R(z)=\frac{R(p_t+(1-p_t)z)}{R(p_t)}$.
Thus, for $\ell(z)=\ln(R(1-z))$, we have $\ell(1-z)+\ell(1-p_t)=\ell\big((1-p_t) (1-z)\big)$ as $p_t+(1-p_t)z=1-(1-p_t)(1-z)$. 

\medskip
\noindent
Therefore, $\,\ell\big((1-p_t) x\big)=\ell(x)+\ell(1-p_t)$. Let $r(y)=\ell\big(e^y\big)$, then 
\begin{equation}\label{eqn:A}
r(y+\varepsilon_t)=r(y)+r(\varepsilon_t) \qquad \forall t\geq 0,
\end{equation}
where $\varepsilon_t=\ln(1-p_t)$. Here $r(0)=\ln R(0) =0$.

\medskip
\noindent
We notice that the domain of $r(y)$ is $y\in (-\infty,0]$, and
$$\{\varepsilon_t \,:\, t \in [0,\infty)\}= (-\infty,0]$$ 
as $1-p_t$ is an increasing function of $t$, mapping $[0,\infty)$ onto $[0,1)$.
Hence, equation \eqref{eqn:A} implies the following Cauchy's Functional Equation
\begin{equation}\label{eqn:B}
r(y+\varepsilon)=r(y)+r(\varepsilon) \qquad \forall y,\varepsilon \in (-\infty,0].
\end{equation}
The general Cauchy's Functional Equation states that assuming 
\begin{itemize}
  \item continuity: $f(x) \in C(\mathbb{R})$,
  \item additivity: $\,f(x+y)=f(x)+f(y)$ for all $x,y \in \mathbb{R}$,
\end{itemize}
the function $f(x)=cx$ for some $c\in \mathbb{R}$.
Notice that \eqref{eqn:B} is a sub-case of the general Cauchy's Functional Equation restricted to a half-line, and therefore has the same linear solution
and the same proof. Thus, \eqref{eqn:B} yields that $r(y)=\kappa y$ for some constant $\kappa$.

\medskip
\noindent
Thus, we have $\ell(x)=\kappa\ln(x)$,
$$\kappa\ln(1-z)=\ell(1-z)=\ln R(z)=\ln\left(\frac{Q(z)-z}{q_0}\right)$$
and
$$Q(z)=z+q_0(1-z)^\kappa$$
Finally, $\,q_1=0\,$ yields $\,Q'(0)=0$. Therefore,  $\,Q'(z)=1-q_0 \kappa(1-z)^{\kappa-1}\,$ implies $\,\kappa={1 \over q_0}$.
\end{proof}

\subsection{Invariant Galton-Watson trees $\mathcal{IGW}(q)$ as attractors}\label{sec:ProofsAttractors}

First we prove the following result, related to Lemma \ref{lem:dSof1}.
\begin{lem}\label{lem:dgLn}
Consider a critical Galton-Watson measure $\mathcal{GW}(\{q_k\})$ with $q_1=0$.
If Assumption \ref{asm:reg} is satisfied, then for $g(x)$ defined in \eqref{eqn:gxDef} 
the following limit 
\be\label{eqn:dgLnLimL}
\lim\limits_{x \rightarrow 1-}{(1-x)g'(x) \over g(x)}
\ee
exists and and is equal to the limit $L$, defined in \eqref{eqn:gLnLimL}.
\end{lem}
\begin{proof}
Note that for $x \in (-1,1)$,
$${Q(x)-x \over (1-x)\big(1-Q'(x)\big)}={1 \over 2-{(1-x)g'(x) \over g(x)}}.$$
Thus, by Assumption \ref{asm:reg}, either the limit $\lim\limits_{x \rightarrow 1-}{(1-x)g'(x) \over g(x)}$
exists or is equal to $\,\pm\infty$. Hence, by the L'H\^{o}pital's rule,
$$L=\lim\limits_{x \rightarrow 1-}\left({\ln{g(x)} \over -\ln(1-x)}\right)=\lim\limits_{x \rightarrow 1-}{(1-x)g'(x) \over g(x)}.$$
\end{proof}

\medskip
\noindent
Before proving Theorem \ref{thm:IGWattractor}, we will need the following result.
\begin{lem}\label{lem:gRegVar}
Consider a critical Galton-Watson measure $\mathcal{GW}(\{q_k\})$ with $q_1=0$,
and let $g(x)$ be as defined in \eqref{eqn:gxDef} and $L$ be as defined in \eqref{eqn:gLnLimL}.
If Assumption \ref{asm:reg} is satisfied, then $g(1-1/y)$ is a regularly varying function (Def. \ref{def:regvar}) with index $L$,
i.e.,
\be\label{eqn:gRegVar}
\lim\limits_{x \to 1-}{g\left(\left(1-{1 \over r}\right)+{1 \over r}x\right) \over g(x)}=\lim\limits_{y \to \infty}{g\left(1-{1 \over ry}\right) \over g\left(1-{1 \over y}\right)}=r^L \qquad \text{ for all }\, r>0.
\ee
\end{lem}
\begin{proof}
For $\alpha>-L-1$, the L'H\^{o}pital's rule and Lemma \ref{lem:dgLn} yield
\begin{align*}
\lim\limits_{y \to \infty}{y^{\alpha+1}g(1-1/y) \over  \int\limits_a^y s^\alpha g(1-1/s)\,ds}
&=\lim\limits_{y \to \infty}{(\alpha+1)y^\alpha g(1-1/y)+ y^{\alpha-1}g'(1-1/y) \over  y^\alpha g(1-1/y)} \\
&=\alpha+1+\lim\limits_{y \to \infty}{y^{\alpha-1}g'(1-1/y) \over  y^\alpha g(1-1/y)} \\
&=\alpha+1+\lim\limits_{x \rightarrow 1-}{(1-x)g'(x) \over g(x)} ~=\alpha+1+L.
\end{align*}
Hence, by the Converse Karamata's theorem (Thm.~\ref{thm:convKaramata}), $g(1-1/y)$ is a regularly varying function with index $L$,
and \eqref{eqn:gRegVar} holds.
\end{proof}

\medskip
\noindent
The following lemma will be the instrument for establishing $\mathcal{IGW}(q)$ trees are attractors.
\begin{lem}\label{lem:IGWattractor}
Consider a Galton-Watson measure $\mathcal{GW}(\{q_k\})$ with $q_1=0$ on $\cT^|$.
Suppose the measure is critical and Assumption \ref{asm:reg} is satisfied. Then, its progeny generating function $Q(z)$ satisfies
$$\lim\limits_{x \rightarrow 1-}{Q\big(z+(1-z)x\big)-\big(z+(1-z)x\big) \over (1-x)(1-Q'(x))}={1 \over 2-L}(1-z)^{2-L},$$
where $L$ is as defined in \eqref{eqn:gLnLimL}.

If the Galton-Watson measure $\mathcal{GW}(\{q_k\})$ (with $q_1=0$) is subcritical, then 
$$\lim\limits_{x \rightarrow 1-}{Q\big(z+(1-z)x\big)-\big(z+(1-z)x\big) \over (1-x)(1-Q'(x))}=1-z.$$
\end{lem}
\begin{proof}
Consider a critical Galton-Watson measure $\mathcal{GW}(\{q_k\})$ with $q_1=0$ and progeny generating function $Q(z)$.
For $x,z \in (-1,1)$, we have
$$Q\big(z+(1-z)x\big)-\big(z+(1-z)x\big)=(1-z)^2 (1-x)^2\, g\big(z+(1-z)x\big)$$
Thus, as
$$1-Q'(x)=2(1-x)g(x)-(1-x)^2g'(x),$$
Lemma \ref{lem:dgLn} yields
\begin{align*}
\lim\limits_{x \rightarrow 1-}&{Q\big(z+(1-z)x\big)-\big(z+(1-z)x\big) \over (1-x)(1-Q'(x))} 
= (1-z)^2 \lim\limits_{x \rightarrow 1-}{g\big(z+(1-z)x\big) \over 2g(x)-(1-x)g'(x)} \\
&= (1-z)^2 \lim\limits_{x \rightarrow 1-}{g\big(z+(1-z)x\big) \over \left(2-{(1-x)g'(x) \over g(x)}\right)g(x)} 
~= {1 \over 2-L}(1-z)^2 \lim\limits_{x \rightarrow 1-}{g\big(z+(1-z)x\big) \over g(x)} \\
&= {1 \over 2-L}(1-z)^2 (1-z)^{-L} ~= {1 \over 2-L}(1-z)^{2-L}
\end{align*}
by \eqref{eqn:gRegVar} with $r={1 \over 1-z}$.
The main statement in Lemma \ref{lem:IGWattractor} follows.

\medskip
\noindent
Now, if we consider a subcritical Galton-Watson measure $\mathcal{GW}(\{q_k\})$ with $q_1=0$ and progeny generating function $Q(z)$.
Then, $Q'(1)<1$, and
\begin{align*}
\lim\limits_{x \rightarrow 1-}&{Q\big(z+(1-z)x\big)-\big(z+(1-z)x\big) \over (1-x)(1-Q'(x))} 
={1 \over 1-Q'(1)} \lim\limits_{x \rightarrow 1-}{Q\big(z+(1-z)x\big)-\big(z+(1-z)x\big) \over 1-x} \\
&={1-z \over 1-Q'(1)} \lim\limits_{x \rightarrow 1-}{Q\big(1-(1-z)(1-x)\big)-Q(1)+(1-z)(1-x) \over (1-z)(1-x)} 
~=1-z.
\end{align*} 
\end{proof}

\medskip
\noindent
Now, we are ready to prove Theorem \ref{thm:IGWattractor}.
\begin{proof}[Proof of Theorem \ref{thm:IGWattractor}.]
Suppose $\mu \equiv\mathcal{GW}(\{q_k\}, \lambda)$ with $q_1=0$  is critical and Assumption \ref{asm:reg} holds.
Then, by equation \eqref{eqn:pruningG} in Lemma~\ref{lem:pruningGW} and Lemma \ref{lem:IGWattractor}, 
the generating function of $\,\textsc{shape}(\S(T))$ converges to
\begin{align*}
z+  \lim\limits_{t \to \infty} {Q\big(z+(1-z)(1-p_t)\big)-(1-p_t) -zp_t \over  p_t\big(1-Q'(1-p_t)\big)}
&=z+\lim\limits_{x \rightarrow 1-}{Q\big(z+(1-z)x\big)-\big(z+(1-z)x\big) \over (1-x)(1-Q'(x))} \\
&=z+{1 \over 2-L}(1-z)^{2-L},
\end{align*}
the generating function of $\mathcal{IGW}(q)$ with $q={1 \over 2-L}$ and $L$ as defined in \eqref{eqn:gLnLimL}.

\medskip
\noindent
If the Galton-Watson measure $\mu \equiv\mathcal{GW}(\{q_k\}, \lambda)$ (with $q_1=0$) is subcritical, then 
Lemma~\ref{lem:pruningGW} and Lemma \ref{lem:IGWattractor} yield convergence of the generating function of $\,\textsc{shape}(\S(T))$ 
\begin{align*}
z+  \lim\limits_{t \to \infty} {Q\big(z+(1-z)(1-p_t)\big)-(1-p_t) -zp_t \over  p_t\big(1-Q'(1-p_t)\big)}
&=z+\lim\limits_{x \rightarrow 1-}{Q\big(z+(1-z)x\big)-\big(z+(1-z)x\big) \over (1-x)(1-Q'(x))} \\
&=z+(1-z) ~=1
\end{align*}
the generating function of $\mathcal{GW}(q_0\!=\!1)$.
Hence,  for $\,T \stackrel{d}{\sim} \mu$, the conditional distribution
${\sf P}\big(\textsc{shape}(\S(T))=\cdot \,\big|\,\S(T)\not=\phi \big)$ converges to a point mass measure, $\mathcal{GW}(q_0\!=\!1)$.
\end{proof}

\medskip
\noindent
\begin{proof}[Proof of Theorem \ref{thm:coloringIGWattractor}.]
Analogously to the proof of Theorem \ref{thm:IGWattractor} above, 
Theorem \ref{thm:coloringIGWattractor} follows from formula \eqref{eqn:coloring} in Theorem \ref{thm:coloring} 
and Lemma \ref{lem:IGWattractor}.
\end{proof}

\medskip
\noindent
\begin{proof}[Proof of Theorem \ref{thm:hrIGWattractor}.]
Following the steps in the proof of Theorem \ref{thm:IGWattractor} above, 
Theorem \ref{thm:hrIGWattractor} follows from formula \eqref{eqn:hrGW} in Theorem \ref{thm:hrGW} 
with $p={\sf P}\big(R_{H_n}\circ \hdots \circ R_{H_1}(T) \not= \phi \big)$
and Lemma \ref{lem:IGWattractor}.
\end{proof}

\section{Discussion}\label{sec:discuss}
In this work, we established the metric and attractor properties 
of the IGW branching processes with respect to a family of the generalized dynamical pruning 
operators.
Informally, these operators eliminate a tree from leaves toward the root and are 
flexible enough to accommodate for a number of classic (e.g. continuous erasure) and custom 
(e.g., erasure by the number of leaves) tree elimination rules.  
Together with the richness of the IGW family, which includes power-law offspring distributions with
tail indices in the range between $1$ and $2$, this makes the presented results 
a useful tool for a variety of physical and mathematical problems.

Observe that erasing a random tree in accordance with the generalized dynamical pruning
describes a coalescence dynamics -- merging of particles represented by the tree
leaves into consecutively larger clusters represented by the internal vertices. 
The invariance and attractor properties of the IGW branching processes can be used 
to study a number of merger dynamics.
For example, the continuum ballistic annihilation process 
(a ballistic motion of random-velocity particles that annihilate at contact) has been shown in \cite{KZ20} to correspond to a generalized dynamical pruning with $\varphi(T) = \textsc{length}(T)$,
as in Example~\ref{ex:L}. 
The invariance of the critical binary Galton-Watson measure $\mathcal{IGW}(1/2, \lambda)$ under the generalized dynamical pruning was used in \cite{KZ20} to obtain an explicit analytical description of 
the annihilation dynamics for a special case of the initial two-valued velocity alternating at the instances of a Poisson process. 
Similarly, the generalized dynamical pruning with $\varphi(T) = \textsc{height}(T)$ as in Example~\ref{ex:height}
corresponds to one dimensional Zeldovich model in cosmology.
The invariance results of this work may provide an interesting analytical insight into the dynamics
of these and other models of coalescence.

The IGW branching processes naturally arise in seismological data
that are traditionally modeled by branching processes with immigration; see \cite{KZBZ21} for a review and discussion.
In essence, a sequence of earthquakes in a region is represented as a collection of clusters, each of which is initiated by an immigrant (the first cluster event).  
It has been shown in \cite{KZBZ21} that the IGW process provides a close approximation to the 
existing earthquake occurrence models and to the observed earthquake cluster statistics in southern California. 
The metric properties of the IGW trees have a meaningful interpretation in the seismological setting, where
the edge length represent interevent times.
The attraction property of the IGW processes allows one to construct a robust earthquake 
modeling framework, which is stable with respect to various catalog uncertainties.
The IGW processes may provide a useful model in other areas that deal with imprecisely
observed data represented by trees. 

We conclude with an open problem.
Lemma~\ref{lem:onlyIGW} established uniqueness of the IGW processes as the invariants
of the generalized dynamical pruning.
The results of Duquesne and Winkel \cite{DW2007} (see Thm.~\ref{thm:coloring} of this paper)
show that the IGW processes are invariant with respect to a broader set of tree reductions. 
It would be interesting to describe all tree transforming operators that preserve
the IGW invariance (and attraction) property.

\section*{Acknowledgements}
We are thankful to Matthias Winkel for pointing to us the connection between the generalized dynamical pruning and the hereditary reduction.
This research is supported by NSF awards DMS-1412557 (YK) and EAR-1723033 (IZ).


\appendix

\section{Lagrange Inversion Theorem}\label{appndx:LIT}

Lagrange Inversion Theorem (aka Lagrange Inversion Formula) can be found in E.\,T. Whittaker and G.\,N. Watson \cite{WW1927}
and in M.~Abramowitz and  I.\,A.~Stegun \cite{AS1964}.
\begin{thm}[{{\bf Lagrange Inversion Theorem}}]\label{thm:LIT}
Consider a function $g:\mathbb{C}\to\mathbb{C}$ such that  $g(w)$ is analytic in a neighborhood of the origin, $g(0)=0$, and $g'(0)\not=0$. 
Then, $g^{-1}$ can be expressed as the following power series near the origin
$$g^{-1}(z)=\sum\limits_{n=1}^\infty {z^n \over n!} \left[{d^{n-1} \over dw^{n-1}}\left({w \over g(w)}\right)^n\right]_{w=0}.$$
Moreover, for any $\varphi:\mathbb{C}\to\mathbb{C}$ analytic in a neighborhood around the origin,
$$\varphi\big(g^{-1}(z)\big)=\varphi(0)+\sum\limits_{n=1}^\infty {z^n \over n!} \left[{d^{n-1} \over dw^{n-1}}\left(\varphi'(w){w \over g(w)}\right)^n\right]_{w=0}.$$
\end{thm}

\section{Regularly varying functions}\label{appndx:Karamata}

We define regularly varying functions and state Karamata's theorems. 
See \cite{RegularBook1989} for a rigorous treatment of the theory of regularly varying functions.
\begin{Def}\label{def:regvar}
A positive measurable function $f(x)$ is said to be {\bf regularly varying} with index $\beta \in\mathbb{R}$ if 
$$\lim\limits_{x \to \infty}{f(rx) \over f(x)}=r^\beta \qquad \text{ for all }\, r>0.$$
\end{Def}

\medskip
\noindent
\begin{thm}[{{\bf Karamata's theorem, direct part \cite{RegularBook1989}}}]\label{thm:Karamata}
Let $f(x):[a,\infty) \to [a,\infty)$ be a regularly varying function with index $\beta \in\mathbb{R}$. Then, 
  $$\lim\limits_{x \to \infty}{x^{\alpha+1}f(x) \over  \int\limits_a^x y^\alpha f(y)\,dy}=\alpha+\beta+1 \qquad \text{ for all }\,\alpha>-\beta-1$$
and  
  $$\lim\limits_{x \to \infty}{x^{\alpha+1}f(x) \over  \int\limits_x^\infty y^\alpha f(y)\,dy}=-(\alpha+\beta+1) \qquad \text{ for all }\,\alpha<-\beta-1.$$
\end{thm}

\medskip
\noindent
We will use the following converse to the above Karamata's theorem.
\begin{thm}[{{\bf Karamata's theorem, converse part \cite{RegularBook1989}}}]\label{thm:convKaramata}
Let $f(x)$ be a positive, measurable, and locally integrable function on $[a,\infty)$ and $\beta \in\mathbb{R}$, then 
\begin{description}
  \item[(a).] If there exist $\alpha>-\beta-1$ such that 
  $$\lim\limits_{x \to \infty}{x^{\alpha+1}f(x) \over  \int\limits_a^x y^\alpha f(y)\,dy}=\alpha+\beta+1,$$
  then $f(x)$ is a regularly varying function with index $\beta$.
  \item[(b).] If for some $\alpha<-\beta-1$, 
  $$\lim\limits_{x \to \infty}{x^{\alpha+1}f(x) \over  \int\limits_x^\infty y^\alpha f(y)\,dy}=-(\alpha+\beta+1),$$
  then $f(x)$ is a regularly varying function with index $\beta$.
\end{description}
\end{thm}

\section{Proof of Lemma \ref{lem:pruningGW}}\label{appndx:ProofsGDPGW}

\begin{proof}[Proof of Lemma \ref{lem:pruningGW}]
First, we show that the tree $\S(\varphi,T)$ obtained by pruning a Galton-Watson tree $T \stackrel{d}{\sim} \mu\equiv\mathcal{GW}(\{q_k\},\lambda)$, is also distributed as 
a Galton-Watson tree.

\medskip
\noindent
For $T \stackrel{d}{\sim} \mu$ and $s \geq 0$, let $T|_{\leq s}$ denote a subtree of $T$ consisting of all points $x$ in the metric space $T$ of distance no
more than $s$ from the root $\rho$, i.e., 
$$T|_{\leq s}=\{x \in T\,:\, d(x,\rho)\leq s\}.$$
Let $T|_{=s}$ denote the points in $T$ of distance $s$ from the root $\rho$, i.e., 
$$T|_{=s}=\{x \in T\,:\, d(x,\rho)=s\}.$$
Let $\cF_s^0=\sigma\big(T|_{\leq s}\big)$ be a sigma algebra generated by the history up to time $s$ (including branching history) of the Galton-Watson process that induces $T$. The future of the Galton-Watson process after time $s$ consists of the descendant subtrees
$$\big\{\Delta_{x,T}\,:x \in T|_{=s} \big\}.$$
Let $\cF_s'=\sigma\big(\Delta_{x,T}\,:x \in T|_{=s}\big)$ be a sigma algebra generated by the future events, after time $s$.
Measure $\mu$ being a Galton-Watson measure ( i.e., $\mu\equiv\mathcal{GW}(\{q_k\},\lambda)$) is equivalent to $T|_{=s}$ being 
a continuous time Markov branching process (see \cite{AN_book,Harris_book}). That is, there exists a filtration $\cF_s \supset \cF_s^0$ such that
\begin{enumerate}
 \item Markov property is satisfied:
  $$P\big(A\,\big|\,\cF_s \big) \,=\, P\big(A\,\big|~T|_{=s}\big) \qquad \forall A \in \cF_s'.$$
  \item  
  Conditioned on $T|_{=s}$, the subtrees $\big\{\Delta_{x,T}\,:x \in T|_{=s} \big\}$ of $T$, denoted here by 
  $$\Big(\big\{\Delta_{x,T}\,:x \in T|_{=s} \big\} ~~\Big|~~ T|_{=s}\Big),$$
  are independent.
\end{enumerate}
Next, let 
$$\mathcal{I}_s=\sigma\big(\Delta_{x,\S(\varphi,T)}: x \in \S(\varphi,T)|_{=s}\big)$$ 
be a sigma algebra generated by the future events of $\,\S(\varphi,T)|_{=s}$, after time $s$. Then, since 
\begin{align*}
\big\{\Delta_{x,\S(\varphi,T)}\,:x \in \S(\varphi,T)|_{=s} \big\}&=\big\{\S(\varphi,\Delta_{x,T})\,:x \in \S(\varphi,T)|_{=s} \big\}\\
&=\big\{\S(\varphi,\Delta_{x,T})\,:x \in T|_{=s}\, \text{ such that }\, \S(\varphi,\Delta_{x,T})\not= \phi  \big\},
\end{align*}
we have $$\mathcal{I}_s =\sigma\big(S(\varphi,\Delta_{x,T})\,: x \in T|_{=s}\big) \, \subset \cF_s'.$$
We claim that conditioned on the event $\{\S(\varphi,T)\not= \phi\}$, the partition/anihilation evolution  
$\,\S(\varphi,T)|_{=s}\,$ is a continuous time Markov branching process with respect to the filtration $\cF_s$. 
Indeed,
\begin{enumerate}
 \item Markov property is satisfied: 
 $$P\big(A\,\big|\,\cF_s \big) \,=\, P\big(A\,\big|~T|_{=s}\big) \,=\, P\big(A\,\big|~\S(\varphi,T)|_{=s}\big) \qquad \forall A \in \cI_s \, \subset \cF_s'.$$
 Let $P_{\not= \phi}(A)=P(A|\,\S(\varphi,T)\not= \phi)$. Then,
 $$P_{\not= \phi}\big(A\,\big|\,\cF_s \big) \,=\, P_{\not= \phi}\big(A\,\big|~\S(\varphi,T)|_{=s}\big) \qquad \forall A \in \cI_s.$$
 
  \item  Conditioned on $\S(\varphi,T)|_{= s}$, the subtrees 
  $$\big\{\Delta_{x,\S(\varphi,T)}\,:x \in \S(\varphi,T)|_{=s} \big\}=\big\{\S(\varphi,\Delta_{x,T})\,:x \in T|_{=s},\, \S(\varphi,\Delta_{x,T})\not= \phi \big\}$$ 
  of $\S(\varphi,T)$  are independent.
 \end{enumerate}

\bigskip
\noindent
In order to characterize the dendritic structure of $\S(\varphi,T)$ we start an upward exploration from the root $\rho \in T$ and proceed to the nearest internal vertex $v$ of $T$ (i.e., ${\sf par}(v)=\rho$). 
For a pair of integers $k \geq 2$ and $m \geq 0$, we have
\be\label{eqn:pBinom}
{\sf  P}\big({\sf br}_T(v)=k, \, {\sf br}_{\S(\varphi,T)}(v)=m \,\big|\,\S(\varphi,T)\not=\phi\big)=\binom{k}{m}(1-p_t)^{k-m} p_t^m  {q_k \over p_t},
\ee
where ${\sf br}_T(v)$ and ${\sf br}_{\S(\varphi,T)}(v)$ denote the branching numbers of vertex $v$ in $T$ and $\S(\varphi,T)$ respectively.
Here, the event ${\sf br}_{\S(\varphi,T)}(v)=1$ corresponds to the case when vertex $v$ is removed due to series reduction.
Thus, the case $m=1$ will be treated separately.

\medskip
\noindent
Next, we would like to find an expression for the branching probability $g_m$ of a pruned tree $\S(\varphi,T)$. 
For a given integer $m \geq 2$,
$${\sf  P}\big({\sf br}_{\S(\varphi,T)}(v)=m \,\big|\,\S(\varphi,T)\not=\phi \big)=(1-p_t)^{-m} p_t^{m-1}  \sum\limits_{k = m}^\infty \binom{k}{m} (1-p_t)^{k}  q_k.$$ 
Therefore, for $m \geq 2$,
\begin{align*}
g_m &= {\sf  P}\big({\sf br}_{\S(\varphi,T)}(v)=m \,\big|\,\S(\varphi,T)\not=\phi, \, {\sf br}_{\S(\varphi,T)}(v)\not= 1 \big) \nonumber \\
&=(1-p_t)^{-m} p_t^{m-1} {\sum\limits_{k = m}^\infty  \binom{k}{m} p^k  q_k \over 1-(1-p_t)^{-1}\sum\limits_{k = 2}^\infty k p^k  q_k}
~={p_t^{m-1} \over m!} Q^{(m)}(1-p_t) \,(1-Q'(1-p_t))^{-1}.
\end{align*}
The corresponding generating function of $\{g_k\}$ is equal to
\begin{align}\label{eqn:genGpk}
G(z) & =\sum\limits_{m=0}^\infty z^m g_m 
=g_0+ {p_t^{-1} \over 1-(1-p_t)^{-1}\sum\limits_{k = 2}^\infty k p^k  q_k}\sum\limits_{m=2}^\infty \sum\limits_{k = m}^\infty  \big(z(1-p_t)^{-1}p_t\big)^m \binom{k}{m} p^k  q_k
\nonumber \\
&= g_0+ {p_t^{-1} \over 1-(1-p_t)^{-1}\sum\limits_{k = 2}^\infty k p^k  q_k} \sum\limits_{k=2}^\infty \sum\limits_{m = 2}^k  \binom{k}{m} \big(z(1-p_t)^{-1}p_t\big)^m p^k  q_k 
\nonumber \\
&= g_0+ {p_t^{-1} \over 1-Q'(1-p_t)}\left(Q\big(z+(1-z)(1-p_t)\big)-Q(1-p_t) -zp_tQ'(1-p_t) \right) \nonumber \\
&=z+g_0+ {Q\big(z+(1-z)(1-p_t)\big)-Q(1-p_t) -zp_t \over  p_t\big(1-Q'(1-p_t)\big)}.
\end{align}
by the binomial theorem. Next, we plug in $z=1$ into \eqref{eqn:genGpk}, obtaining
$$g_0={Q(1-p_t)-(1-p_t) \over p_t\big(1-Q'(1-p_t)\big)}$$
as $G(1)=1$. Hence, \eqref{eqn:genGpk} rewrites as \eqref{eqn:pruningG}.
We proceed by differentiating ${d \over dz}$ in \eqref{eqn:pruningG}, obtaining
\be\label{eqn:genGpkdz}
G' (z)={Q'(1-p_t+zp_t)-Q'(1-p_t) \over 1-Q'(1-p_t)}.
\ee

\medskip
\noindent
An edge in $\S(\varphi,T)$ is a union of edges in the tree obtained after pruning $T$, but before the series reduction, separated by the degree two vertices.
The probability of a degree two vertex after pruning (but before the series reduction) is 
$$\sum\limits_{k=2}^\infty q_k k p_t (1-p_t)^{k-1}=p_tQ'(1-p_t).$$
Hence, by Wald's equation, the edge lengths in $\S(\varphi,T)$ are independent exponential random variables, each with rate
$$\lambda \big(1-Q'(1-p_t)\big).$$

\medskip
\noindent
Finally, we observe that if $\mu(T)$ is critical, then $Q'(1)=1$ and \eqref{eqn:genGpkdz} imply $G' (1)=1$. 
That is, $\nu(T\,|T\not= \phi)$ is a critical Galton-Watson measure.
\end{proof}

\end{document}